\documentclass[11pt,a4paper]{article}
\usepackage[utf8]{inputenc}
\usepackage[a4paper, total={6.6in, 9.6in}]{geometry}

\usepackage[T1]{fontenc}
\usepackage[variablett]{lmodern}
\usepackage{amsmath}
\usepackage{amsfonts}
\usepackage{amssymb}
\usepackage{float}
\usepackage{wrapfig}
\usepackage{enumitem}
\usepackage{wasysym}

\usepackage{listings}
\lstdefinelanguage{GAP}{%
  morekeywords={%
    Assert,Info,IsBound,QUIT,%
    TryNextMethod,Unbind,and,break,%
    continue,do,elif,%
    else,end,false,fi,for,%
    function,if,in,local,%
    mod,not,od,or,%
    quit,rec,repeat,return,%
    then,true,until,while%
  },%
  sensitive,%
  morecomment=[l]\#,%
  morestring=[b]",%
  morestring=[b]',%
}[keywords,comments,strings]
\usepackage[svgnames]{xcolor}
\usepackage{graphicx}
\usepackage{quiver}

\usepackage[backend=bibtex,style=alphabetic]{biblatex}
\bibliography{fielddef}

\usepackage{caption}
\usepackage{hyperref}
\hypersetup{
    colorlinks = true,
}
\PassOptionsToPackage{unicode}{hyperref}
\PassOptionsToPackage{naturalnames}{hyperref}
\usepackage[nameinlink]{cleveref}
\Crefname{subsection}{Subsection}{Subsections}
\Crefname{question}{Question}{Questions}
\Crefname{subsubsection}{Paragraph}{Paragraphs}

\newlist{propenum}{enumerate}{1}
\setlist[propenum]{label=(\roman*), ref=\theproposition~(\roman*)}
\crefalias{propenumi}{proposition} 

\newlist{thmenum}{enumerate}{1}
\setlist[thmenum]{label=(\roman*), ref=\thetheorem~(\roman*)}
\crefalias{thmenumi}{theorem} 

\newlist{corenum}{enumerate}{1}
\setlist[corenum]{label=(\roman*), ref=\thecorollary~(\roman*)}
\crefalias{corenumi}{corollary} 

\usepackage{comment}

\newcommand{\N}{\mathbb{N}}
\newcommand{\Z}{\mathbb{Z}}

\newcommand{\Q}{\mathbb{Q}}
\newcommand{\C}{\mathbb{C}}

\newcommand{\Qbar}{\bar\Q}

\newcommand{\piet}{\pi_1^{\text{ét}}}

\renewcommand{\ni}{\mathrm{ni}}
\newcommand{\nid}{\ni^{\natural}}

\newcommand{\Ima}{\mathrm{Im}}

\renewcommand{\P}{\mathbb{P}}

\newcommand{\PC}{\P^1(\C)}

\newcommand{\Spec}{\mathrm{Spec}}

\newcommand{\Frac}{\mathrm{Frac}}

\newcommand{\Aut}{\mathrm{Aut}}
\newcommand{\Gal}{\mathrm{Gal}}

\newcommand{\Sym}{\mathfrak{S}}
\newcommand{\card}[1]{ \left | #1 \right | }
\newcommand{\gen}[1]{ \left \langle #1 \right \rangle }
\newcommand{\ord}{\mathrm{ord}}

\newcommand{\Hur}{\mathrm{Hur}} 
\newcommand{\Hn}{\mathcal{H}} 
\newcommand{\Conf}{\mathrm{Conf}} 
\newcommand{\B}{\mathrm{B}} 

\newcommand{\Comp}{\mathrm{Comp}} 

\newcommand{\id}{\mathrm{id}}

\newcommand{\gbar}{\underline{g}}
\newcommand{\tbar}{\underline{t}}
\newcommand{\cbar}{\underline{c}}

\newcommand{\eqdef}{\overset{\text{def}}{=}}

\newcommand{\verti}{\, \middle \vert \,}
\newcommand{\bigverti}{\, \bigg\vert \,}

\newcommand{\rmN}{\mathrm{N}}

\renewcommand{\bar}{\overline}
\renewcommand{\hat}{\widehat}
\renewcommand{\tilde}{\widetilde}

\newcommand{\ab}{\text{ab}}

\newcommand{\bigslant}[2]{{\raisebox{.2em}{$#1$}\left/\raisebox{-.2em}{$#2$}\right.}}

\usepackage{amsthm}
\usepackage{thmtools}

\newcounter{mycounter}[section]

\theoremstyle{plain}
\newtheorem{theorem}[mycounter]{Theorem}
\newtheorem{corollary}[mycounter]{Corollary}
\newtheorem{proposition}[mycounter]{Proposition}
\newtheorem{lemma}[mycounter]{Lemma}

\theoremstyle{remark}
\newtheorem{remark}[mycounter]{Remark}

\theoremstyle{definition}
\newtheorem{definition}[mycounter]{Definition}
\newtheorem{question}[mycounter]{Question}
\newtheorem{example}[mycounter]{Example}

\counterwithin{equation}{section}

\usepackage{titletoc}
\titlecontents{section}
[0pt]                                               
{}%
{\contentsmargin{0pt}                               
    \large\thecontentslabel.\enspace
    }
{\contentsmargin{0pt}\large}                        
{\titlerule*[.5pc]{ }\contentspage}                 
[]                                                  

\usepackage{titlesec}
\titleformat{\section}[block]{\normalfont\centering\scshape\large}{\thesection.}{1em}{}
\titleformat{\subsection}[block]{\normalfont\large}{\thesubsection.}{1em}{\bf}
\titleformat{\subsubsection}[runin]{\normalfont}{\bf\thesubsubsection.}{0.3em}{\bf}

\makeatletter
\patchcmd{\@maketitle}{\LARGE}{\huge}{\typeout{OK 1}}{\typeout{Failed 1}}
\patchcmd{\@maketitle}{\large \lineskip}{\Large \lineskip}{\typeout{OK 2}}{\typeout{Failed 2}}
\makeatother

\title{
  Fields of Definition of \\
  Components of Hurwitz Spaces
}

\author{Béranger Seguin\footnote{Univ. Lille, CNRS, UMR 8524 - Laboratoire Paul Painlevé, F-59000 Lille, France. Email: \texttt{beranger.seguin@ens.fr}.}}
\date{}

\renewenvironment{abstract}{%
\hfill\begin{minipage}{0.95\textwidth}
\rule{\textwidth}{1pt} \textsc{Abstract.}}
{\par\noindent\rule{\textwidth}{1pt}\end{minipage}}

\begin{document}

\maketitle{}

\begin{abstract}
  For a fixed finite group $G$, we study the fields of definition of geometrically irreducible components of Hurwitz moduli schemes of marked branched $G$-covers of the projective line.
  The main focus is on determining whether components obtained by ``gluing'' two other components, both defined over a number field $K$, are also defined over $K$.
  The article presents a list of situations in which a positive answer is obtained.
  As an application, when $G$ is a semi-direct product of symmetric groups or the Mathieu group $M_{23}$, components defined over $\Q$ of small dimension ($6$ and $4$, respectively) are shown to exist.

  \bigskip

\end{abstract}

{
  \hypersetup{linkcolor=black}
  \tableofcontents{}
}
\hfill\rule{0.95\textwidth}{1pt}

\section{Introduction}

Let $G$ be a finite group and $K$ be a field of characteristic zero.
Hurwitz schemes are moduli spaces of branched $G$-covers of $\P^1$.
Their $K$-points are particularly significant for number theory since they are tightly related to the inverse Galois problem for $G$ over $K(t)$; see \autocite{rw, Fried77, Fried91}.

When $K$ is algebraically closed, Riemann's existence theorem implies that the theory reduces to topology.
Specifically, $\C$-points of Hurwitz schemes correspond to isomorphism classes of topological $G$-covers of punctured Riemann spheres.
A classical topological construction lets one ``glue'' two marked covers -- one with $r_1$ branch points and one with $r_2$ branch points -- into a single marked cover with $r_1 + r_2$ branch points.
This gluing operation plays a central role in \autocite{EVW}.

When $K$ is a discrete complete valued field, Harbater defined an analogous \emph{patching} operation to construct covers defined over $K$ with a specified monodromy group by patching together covers with smaller monodromy groups; see \autocite{harbater, haran, liu}.
This construction leads to a positive answer to the inverse Galois problem over $K(T)$.

For number theorists, the most interesting case is that of number fields.
However, no gluing or patching operation is known in this case, although it would be a game-changing tool for inverse Galois theory.
In this article, we focus not on $G$-covers themselves but on geometrically irreducible components of Hurwitz moduli schemes.
Identifying components defined over $\Q$ is a crucial first step in finding rational points on these schemes.
For this reason, the question of the fields of definition of these components is a well-studied topic; see \autocite{Cau, DebEms, EVW2, Fried91}.

Assume $K$ is a number field.
The gluing operation over $\bar K$ induces a monoid structure on the set $\Comp(G)$ of geometrically irreducible components of Hurwitz moduli schemes of marked branched $G$-covers of $\P^1$.
In \autocite{geomringcomp}, we studied this product operation and the corresponding monoid ring, introduced in \autocite{EVW} under the name \emph{ring of components}.
To understand the arithmetic properties of the topological gluing operation, a prominent question is the following:

\begin{question}
  \label{qn:prod-defined}
  Let $x,y \in \Comp(G)$ be components defined over $K$.
  Is the component $xy$, obtained by gluing $x$ and $y$, defined over $K$?
\end{question}
\Cref{qn:prod-defined} and related problems are the main focus of this article.
Our main result is that the answer is positive in situations (i), (ii) and (iii) below:
\begin{theorem}
  \label{thm:main}
  Let $x, y \in \Comp(G)$ be components defined over $K$.
  Denote by $H_1$ and $H_2$ the monodromy groups of the covers contained in $x$ and $y$ respectively, and let $H = \gen{H_1, H_2}$.
  Then:
  \begin{thmenum}
    \item
      \label{thmitem:main-prod-permuting-defK}
      If $H_1 H_2 = H$, then the glued component $xy$ is defined over $K$.
    \item
      \label{thmitem:main-prod-Mbig-defK}
      If every conjugacy class of $H$ that is a local monodromy class of the covers in the component $xy$ occurs at at least $M$ branch points (for some constant $M$ which depends only on the group $G$), then $xy$ is defined over $K$.
    \item
    \label{thmitem:main-prod-conj-defK}
      There are elements $\gamma, \gamma' \in H$ satisfying $\gen{H_1^{\gamma}, H_2^{\gamma'}} = H$ such that the component $x^{\gamma} y^{\gamma'}$, obtained by letting $\gamma, \gamma'$ act on $x,y$ and by gluing the resulting components, is defined over $K$.
  \end{thmenum}
\end{theorem}

In \Cref{sn:preliminaries}, we establish the notation and introduce the key objects.
The three parts of \Cref{thm:main} are then proved in three corresponding sections:

\begin{itemize}
  \item
    \Cref{thmitem:main-prod-permuting-defK} (which is \Cref{thmitem:prod-permuting-defK}) is proved using techniques introduced by \autocite{Cau} and lemmas about the braid group action on tuples.
    In \Cref{sn:cau-thm}, we present this result, including a generalized version, and propose applications.
    Cases of interest include the situation where $H_1$ or $H_2$ is normal in $H$, notably if one is included in the other.
    A consequence of our ideas is presented in \Cref{susbn:galois-small}: we show that the Galois action on components is entirely determined by the Galois action on components with few branch points; precise statements are given in \Cref{prop:reduction-galact,prop:small-deg-all-defK}.
    An application of \Cref{thmitem:main-prod-permuting-defK} is given in \Cref{ex:semi-direct-prod}: if $G$ is a semi-direct product of symmetric groups, there is a component defined over $\Q$ of connected $G$-covers with six branch points.
    This improves on a similar example of Cau where twelve branch points were used.
  \item
    \Cref{thmitem:main-prod-Mbig-defK} (which is \Cref{thmitem:prod-Mbig-defK}) is proved using the \emph{lifting invariant}, defined in \autocite{EVW2, wood} after ideas of Fried \autocite{Fried95}, and a version of the Conway-Parker theorem found in these articles.
    In \Cref{sn:liftinv}, we review this invariant and use it to determine the fields of definition of glued components.
  \item
    \Cref{thmitem:main-prod-conj-defK} (which is \Cref{thm:prod-conj-defK}) is based on \emph{patching results} over complete valued fields.
    We follow the algebraic patching approach from \autocite{haran}.
    By patching covers over infinitely many complete valued fields, we obtain a result concerning the fields of definition of components in the non-complete case.
    \Cref{sn:patching} is concerned with the proof of this theorem.
    Application are given in \Cref{ex:m23,ex:17T7}: when $G$ is the Mathieu group $M_{23}$ or the transitive group $\mathrm{PSL}_2(16) \rtimes \Z/2\Z$, there is a component defined over $\Q$ of connected $G$-covers with only four branch points.
\end{itemize}

We do not know if the answer to \Cref{qn:prod-defined} is always positive.
Finding counterexamples is difficult because there are few tools available to prove that components are not defined over $\Q$.
For instance, the lifting invariant cannot be used to find a counterexample, as established by \Cref{thm:galois-gpact-liftinv}.
Indeed, the lifting invariant of a product of components defined over $K$ is invariant under the Galois action: from the point of view of this invariant, products of components defined over $K$ are indistinguishable from components defined over $K$.


\section{Preliminaries}
\label{sn:preliminaries}

In this section, we define key objects and review classical results concerning $G$-covers, their moduli spaces, their components and the Galois action.
In \Cref{subsn:notation}, notational and terminological choices are presented.
In \Cref{subsn:main-objects}, we introduce the main objects of this article, notably $G$-covers and Hurwitz spaces, both topological and algebraic.
Then, a description of the Galois action of $\Gamma_K$ on $G$-covers and their components is given in \Cref{subsn:galact} and used to define fields of definitions of $G$-covers and of components of Hurwitz spaces.
Finally, in \Cref{subsn:bcl}, we give various versions of the \emph{branch cycle lemma}, a fundamental result concerning the Galois action on covers and components.

\subsection{Notation}
\label{subsn:notation}

We use the following notation and terminology throughout the article.
In what follows, $G$ is a finite group and $K$ is a number field.
Number fields are always equipped with an embedding into $\Qbar$.
We denote by $\Gamma_K$ the absolute Galois group $\Gal(\Qbar \mid K)$.
The cyclotomic character is the group morphism $\chi : \Gamma_K \rightarrow \hat \Z^{\times}$ determined by the Galois action on roots of unity: if $\zeta \in \Qbar$ is an $n$-th root of unity and $\sigma \in \Gamma_K$, then $\sigma(\zeta)=\zeta^{\chi(\sigma)\mod n}$.


\subsubsection{Conventions.}
The cardinality of a set $X$ is denoted by $\card{X}$.
We write $g^h=h g h^{-1}$ for conjugation in a group.
We denote by $\ord(g)$ the order of an element $g$ in a finite group $H$.
If $g \in H$ and $\alpha \in \hat \Z$ is a profinite integer, $g^\alpha$ is the well-defined element $g^{\alpha \mod \ord(g)} \in H$.
Similarly, if $c \subseteq H$ is a conjugacy class, the order of its elements is denoted by $\ord(c)$, and $c^{\alpha}$ is the conjugacy class of the $\alpha$-th powers of elements of $c$, where $\alpha$ is either an integer or a profinite integer.

\begin{definition}
  \label{defn:rational-subset}
  A subset $c$ of $G$ is \emph{$K$-rational} if for every $g \in c$ and $\sigma \in \Gamma_K$ we have $g^{\chi(\sigma)} \in c$.
\end{definition}

If $K=\Q$, we have $\Ima(\chi) = \hat\Z^{\times}$.
Therefore, $\Q$-rational subsets are subsets closed under $n$-th powers for all $n$ coprime with $\card{G}$.
In contrast, if $K$ contains all $\card{G}$-th roots of unity, then the image of $\chi$ is trivial modulo $\card{G}$ and every subset of $G$ is $K$-rational.
Examples of sets which are always $K$-rational include $G$, $G \setminus \{1\}$, as well as any subset of $G$ consisting of involutions.

\subsubsection{Tuples.}
Tuples are denoted with underlined roman letters.
Let $\gbar = (g_1,\ldots,g_n)$ be a tuple of elements of $G$.
Then:
\begin{itemize}
  \item
    Its \emph{degree} or \emph{size} $\deg(\gbar)$ is the number $n$ of elements in the tuple.
  \item
    Its \emph{group} $\gen{\gbar}$ is the subgroup of $G$ generated by $g_1,\ldots,g_n$.
    If $\gbar_1, \ldots, \gbar_s$ are tuples, we denote by $\gen{\gbar_1,\ldots,\gbar_s}$ the subgroup of $G$ generated by the subgroups $\gen{\gbar_1}, \ldots, \gen{\gbar_s}$.
  \item
    The \emph{product} of $\gbar$ is $\pi\gbar = g_1 g_2 \cdots g_n \in G$.
    We say that $\gbar$ is a \emph{product-one tuple} if $\pi\gbar = 1$.
  \item
    Let $H$ be a subgroup of $G$ containing $\gen{\gbar}$.
    A conjugacy class $\gamma$ of $H$ \emph{appears in $\gbar$} if there is some $i$ for which $g_i \in \gamma$.
    The set of the conjugacy classes of $H$ which appear in $\gbar$ is denoted by $D_H(\gbar)$.
    We denote by $c_H(\gbar)$ the conjugation-invariant subset of $H$ obtained as the union of all classes in $D_H(\gbar)$.
    If $H$ is not specified in the notation, it is assumed that $H = \gen{\gbar}$.
  \item
    Let $H$ be a subgroup of $G$ which contains $\gen{\gbar}$ and $c$ be a conjugation-invariant subset of $H$ which contains $c_H(\gbar)$.
    If $\gamma$ is a conjugacy class of $H$ contained in $c$, we denote by $\mu_{H,c}(\gbar)(\gamma)$ the count of its appearances in $\gbar$, i.e.:
      \[
        \mu_{H,c}(\gbar)(\gamma)
        =
        \left \vert
          \left \lbrace
            i \in \{1,\ldots,n\}
            \bigverti
            g_i \in \gamma
          \right \rbrace
        \right \vert
        .
      \]
    
    This defines an integer-valued map $\mu_{H,c}(\gbar)$ on the set $D$ of all conjugacy classes of $H$ contained in $c$.
    We call this map the \emph{$(H,c)$-multidiscriminant of $\gbar$}.
    If $c$ is not specified in the notation $\mu_H(\gbar)$, it is assumed that $c = c_H(\gbar)$.
\end{itemize}

\subsubsection{Schemes.}
Let $L$ be a field.
Here, \emph{$L$-schemes} are separated schemes equipped with a morphism into $\Spec(L)$.
Let $L' \mid L$ be a field extension and $X$ be an $L$-scheme of finite type.
We denote by $X_{L'}$ the $L'$-scheme $X \underset{\Spec(L)}{\times} \Spec(L')$ obtained by extending the scalars.
The set $X(L')$ of $L'$-points of $X$ is the set of morphisms of $L$-schemes from $\Spec(L')$ to $X$, or equivalently morphisms of $L'$-schemes from $\Spec(L')$ to $X_{L'}$.
An $L'$-point $x \in X(L')$ is \emph{$L$-rational} if there exists an $L$-point $x' \in X(L)$ such that the following diagram of $L$-schemes commutes:
\[\begin{tikzcd}[ampersand replacement=\&]
	{\Spec(L')} \& {\Spec(L)} \& X
	\arrow[from=1-1, to=1-2]
	\arrow["{x'}", from=1-2, to=1-3]
	\arrow["x"', curve={height=22pt}, from=1-1, to=1-3]
\end{tikzcd}\]
The point $x'$ is called an \emph{$L$-model} of the point $x$.
Similarly, an $L'$-subscheme $Y$ of $X_{L'}$ is \emph{defined over $L$} if there exists an $L$-subscheme $Y'$ of $X$ such that the $L'$-subscheme $Y'_{L'}$ of $X_{L'}$ is equal to $Y$.
We say that $Y'$ is an \emph{$L$-model} of $Y$.

Assume now $L'\mid L$ is Galois.
An automorphism $\sigma \in \Gal(L' \mid L)$ induces an $L$-automorphism of $\Spec(L')$ which we denote by $\Spec(\sigma)$.
The group $\Gal(L' \mid L)$ acts on an $L'$-point $x \in X(L')$ by the formula $\sigma.x = x \circ \Spec(\sigma)$ and on an $L'$-subscheme $Y \subseteq X_{L'}$ by pullback along $\id_X \underset{\Spec(L)}{\times} \Spec(\sigma)$:
\[
  \begin{matrix}
    \begin{tikzcd}[ampersand replacement=\&]
	{\Spec(L')} \\
	{\Spec(L')} \& X
	\arrow["x"', from=2-1, to=2-2]
	\arrow["\Spec(\sigma)"', from=1-1, to=2-1]
	\arrow["{\sigma.x}", from=1-1, to=2-2]
\end{tikzcd}
&
&
\begin{tikzcd}[ampersand replacement=\&]
	{\sigma.Y} \& {X_{L'}} \\
	Y \& {X_{L'}}
	\arrow[from=1-1, to=2-1]
	\arrow[from=1-1, to=1-2]
	\arrow["\id_X \underset{\Spec(L)}{\times} \Spec(\sigma)", from=1-2, to=2-2]
	\arrow[from=2-1, to=2-2]
	\arrow["\bigg\lrcorner"{anchor=center, pos=0.125}, draw=none, from=1-1, to=2-2]
\end{tikzcd}
\end{matrix}
.\]

\begin{proposition}[Galois descent]
  \label{prop:galois-descent}
  The following equivalences hold:
  \begin{itemize}
    \item
      An $L'$-point of $X$ is $L$-rational if and only if it is invariant under the action of $\Gal(L'\mid L)$.
    \item
      An $L'$-subscheme $Y$ of $X_{L'}$ is defined over $L$ if and only if it is globally preserved by the action of $\Gal(L'\mid L)$, i.e. for every $\sigma \in \Gal(L'\mid L)$ there is an $L'$-automorphism $\sigma'$ of $Y$ such that the following diagram commutes:
      \[\begin{tikzcd}[ampersand replacement=\&]
        Y \& {X_{L'}} \\
        Y \& {X_{L'}}
        \arrow["{\sigma'}"', from=1-1, to=2-1]
        \arrow["\subseteq", hook, from=1-1, to=1-2]
        \arrow["\subseteq", hook, from=2-1, to=2-2]
        \arrow["\id_X \underset{\Spec(L)}{\times} \Spec(\sigma)", from=1-2, to=2-2]
      \end{tikzcd}.\]
    \item
      If $L'$ is algebraically closed and $Y$ is a reduced $L'$-subscheme of $X_{L'}$, then $Y$ is defined over $K$ if and only if the subset $Y(L')$ of $X(L')$ is globally preserved by the action of $\Gal(L'\mid L)$.
  \end{itemize}
\end{proposition}

\subsection{Main objects}
\label{subsn:main-objects}:

In this section, we define the main objects which are used in the text.
Configuration spaces (\Cref{subsubsn:conf}), $G$-covers (\Cref{subsubsn:top-gcovers,subsubsn:algcovers}) and Hurwitz spaces (\Cref{subsubsn:top-hur,subsubsn:hur-sch}) are introduced, both in the topological and in the algebraic settings, and the links between the two contexts are explicited.
The combinatorial description of $G$-covers and their components is also stressed (\Cref{subsubsn:top-hur,subsubsn:monoid-comp,subsubsn:braidact}).

\subsubsection{Configurations and braid groups.}
\label{subsubsn:conf}

A \emph{configuration} $\tbar = \{t_1,\ldots,t_n\}$ is an unordered list of $n$ distinct complex numbers.
Configurations form a topological space $\Conf_n(\C)$, whose topology is inherited from the standard topology on $\C^n$ after removing tuples with equal elements and quotienting out by the action of $\Sym_n$.
The fundamental group of $\Conf_n(\C)$ is the \emph{Artin braid group} $\B_n$.
It admits a simple presentation: generators are given by the \emph{elementary braids} $\sigma_1, \ldots, \sigma_{n-1}$ satisfying the following generating set of relations:
\begin{itemize}
  \item
    $\sigma_i \sigma_j = \sigma_j \sigma_i$ for all $i,j \in \{1,\ldots,n-1\}$ satisfying $|i-j| > 1$;
  \item
    $\sigma_i \sigma_{i+1} \sigma_i = \sigma_{i+1} \sigma_i \sigma_{i+1}$ for all $i \in \{1,\ldots,n-2\}$.
\end{itemize}

A configuration $\tbar \in \Conf_n(\C)$ is \emph{defined over $K$} if the elements $t_1,\ldots,t_n$ are all algebraic and are permuted by the Galois action of $\Gal(\Qbar\mid K)$.
We denote by $\Conf_n(K)$ the set of configurations of $\Conf_n(\C)$ defined over $K$.

The space of configurations has a scheme counterpart.
Indeed, fixing a configuration $\tbar = \{t_1,\ldots,t_n\}$ amounts to fixing the monic polynomial $(X-t_1)\cdots(X-t_n)$, of degree $n$ with no double roots; we can parametrize these polynomials by their coefficients instead of their roots: the scheme $\Conf_n$ is the open subvariety of $\mathbb{A}^n$ obtained by removing the closed Zariski subset $\Delta$ defined by the polynomial equation ``the discriminant of $X^n + a_1 X^{n-1} + \ldots + a_{n-1} X + a_n$ cancels''.
The $K$-points of $\Conf_n$ are in natural bijection with the configurations of $n$ points defined over $K$, and its $\C$-points are precisely the elements of $\Conf_n(\C)$, which makes the notation unambiguous.

\subsubsection{Topological $G$-covers and tuples.}
\label{subsubsn:top-gcovers}

In this article, we consider branched $G$-covers of the projective line.
Let $\tbar \in \Conf_n(\C)$ be a configuration.
Topological $G$-covers branched at $\tbar$ are covering maps $p:X \to \P^1(\C)\setminus\tbar$ equipped with a morphism $G \to \Aut(p)$ that induces a free transitive action on each fiber.
Note that $G$-covers are necessarily Galois covers of degree $\card{G}$.
We do not assume that these covers are connected, and we allow trivial ramification at the ``branch points'', even if the $G$-cover can be extended into a $G$-cover with less branch points.

A \emph{marked $G$-cover} is a $G$-cover equipped with a marked point in the unramified fiber above the point at infinity $\infty$.
The monodromy based at $\infty$ associates to every marked $G$-cover $p$ branched at $\tbar$ a morphism:
  \[ \pi_1(\P^1(\C)\setminus\{t_1,\ldots,t_n\}, \infty) \rightarrow G \]
which uniquely characterizes the isomorphism class of the marked $G$-cover $p$.

Choose a topological bouquet $\gamma_1, \ldots, \gamma_n$, as defined in \autocite[Paragraph 1.1]{DebEms}.
This is a set of generators of $\pi_1(\P^1\setminus\{t_1,\ldots,t_n\}, \infty)$, where $\gamma_i$ is the homotopy class of a loop which rotates once counterclockwise around $t_i$.
The relations between these generators are generated by the single relation $\gamma_1\cdots\gamma_n=1$.
The choice of a bouquet induces a bijection between isomorphism classes of marked $G$-covers branched at $\tbar$ and $n$-tuples $\gbar = (g_1,\ldots,g_n)$ of elements of $G$ which are \emph{product-one}, i.e. such that $\pi\gbar = g_1 \cdots g_n = 1$.
We say that $(g_1,\ldots,g_n)$ is the \emph{branch cycle description} of the marked $G$-cover.
Via this description, we have a dictionary between geometric and group-theoretic/combinatorial objects:
\begin{itemize}
  \item
    The monodromy group of the marked $G$-cover, i.e. the automorphism group of the connected component of its marked point, is $\gen{\gbar} = \gen{g_1,\ldots,g_n}$.
    In particular, the $G$-cover is connected exactly when $g_1,\ldots,g_n$ generate $G$.
  \item
    unmarked $G$-covers correspond to orbits of product-one tuples under the conjugation action of $G$:
      \[ (g_1,\ldots,g_n)^{\gamma} = (g_1^{\gamma},\ldots,g_n^{\gamma}) \text{ for } \gamma \in G. \]
    This conjugation action corresponds to the free transitive action of $G$ on the fiber above the basepoint $\infty$.
    It amounts to a change of marked point. 
\end{itemize}

\begin{remark}
  We include non-connected $G$-covers, i.e. covers whose monodromy groups are proper subgroups of $G$, because we are interested in patching-like results.
  Typically, we want to construct components with monodromy group $G$ by gluing components with smaller monodromy groups.
  If we do not take this phenomenon into account, the answer to \Cref{qn:prod-defined} is ``yes'': the concatenation of two components defined over $K$ of connected $G$-covers is always defined over $K$.
  This follows from \Cref{thmitem:main-prod-permuting-defK}.
  In \autocite{Cau}, a different but equivalent choice is made: instead of considering components of marked non-connected $G$-covers, Cau considers components of unmarked connected $H$-covers where $H$ is a subgroup of $G$.
  The links between these two approaches are discussed further in \Cref{subsubsn:marked-unmarked}.
\end{remark}

\subsubsection{Topological Hurwitz spaces and their components.}
\label{subsubsn:top-hur}

Unless specified otherwise, Hurwitz spaces in this article are moduli spaces of \emph{marked} $G$-covers, connected or not.
We denote by $\Hur^*(G,n)$ the topological Hurwitz space of marked $G$-covers with $n$ branch points.
It is a covering space of $\Conf_n(\C)$ whose fiber above some configuration $\tbar \in \Conf_n(\C)$ consists of isomorphism classes of marked $G$-covers of $\PC \setminus \tbar$.

Classically, there is an action of the braid group $\B_n$ on $n$-tuples of elements of $G$, induced by the following formula:
\[
  \sigma_i.(g_1, \ldots, g_n) = (g_1,\ldots, g_{i-1}, g_{i+1}^{g_i}, g_i, \ldots, g_n).
\]
We say that two $n$-tuples of elements of $G$ are \emph{braid equivalent} when they are in the same $\B_n$-orbit for this action.
The connected components of $\Hur^*(G,n)$ are in bijection with braid group orbits of product-one $n$-tuples $\gbar = (g_1,\ldots,g_n)$ of elements of $G$.
Whereas the branch cycle description $\gbar$ of a cover depends on the choice of a bouquet, the braid group orbit of $\gbar$ does not: this is due to the fact that $\B_n$ acts transitively on topological bouquets up to conjugation \autocite{DebEms}.
Thus, there is a \emph{canonical} bijection between $\B_n$-orbits of product-one $n$-tuples of elements of $G$ and connected components of $\Hur^*(G,n)$.

\subsubsection{The monoid of components.}
\label{subsubsn:monoid-comp}

Via the description of connected components of $\Hur^*(G,n)$ as $\B_n$-orbits of product-one $n$-tuples of elements of $G$, the concatenation of tuples induces a well-defined product operation on components of Hurwitz spaces:
  \[
    (g_1,\ldots,g_n)(g'_1,\ldots,g'_{n'})=(g_1,\ldots,g_n,g'_1,\ldots,g'_{n'}).
  \]
The graded \emph{monoid of components} $\Comp(G)$ is:
\[
  \bigsqcup_{n \geq 0}
  \left (
    \bigslant
    {
    \left \lbrace
      \gbar \in G^n
      \verti
      \pi\gbar = 1
    \right \rbrace
    }
    {\B_n}
  \right ),
\]
graded by the size $n$ of a tuple, and equipped with the product operation induced by concatenation.
Elements of degree $n$ of $\Comp(G)$ are in bijection with connected components of $\Hur^*(G,n)$.\footnote{
  For us, the \emph{degree} of a component is its degree as an element of the graded monoid $\Comp(G)$, i.e. the number of branch points of the covers it contains, which is the size of its representing tuples.
  This is also equal to the \emph{dimension} of the component since $\Hur^*(G,n)$ is a finite cover of the $n$-dimensional space $\Conf_n(\C)$.
}
The identity element of $\Comp(G)$ is the braid orbit of the empty tuple, i.e. the connected component of the trivial $G$-cover.
The monoid $\Comp(G)$ is commutative and finitely generated, cf. \autocite[Subsection 3.2]{geomringcomp}.
We often abusively refer to an element of $\Comp(G)$ by one of its representing tuples $\gbar \in G^n$.

\subsubsection{The braid group action.}
\label{subsubsn:braidact}

We recall results about the braid group action on tuples:
\begin{proposition}
  \label{prop:braid-results}
  ~
  \begin{propenum}
    \item
      The product, group, multidiscriminant of a tuple depend only on its braid group orbit.
    \item
      The braid group orbit of a concatenation of tuples depends only on their braid group orbits.
    \item
      If $\gbar$ and $\gbar'$ are product-one tuples, then the concatenated tuples $\gbar\gbar'$ and $\gbar'\gbar$ are braid equivalent.
      
    \item
      \label{propitem:braid-conjugate}
      If $\gbar_1, \gbar_2, \gbar_3$ are product-one tuples, then $\gbar_1\gbar_2\gbar_3$ is braid equivalent to $\gbar_1\gbar_2^{\gamma}\gbar_3$ for any $\gamma$ which is either in $\gen{\gbar_1, \gbar_3}$ or in $\gen{\gbar_2}$.

  \end{propenum}
\end{proposition}

Proofs can be found with this notation in \autocite{geomringcomp}, or with different notation in \autocite{Cau}.
Together, points (i), (ii) and (iii) ensure that the monoid of components $\Comp(G)$ is a well-defined commutative monoid, and we can talk about the group and multidiscriminant of a component $x \in \Comp(G)$.
We generalize the use of the symbols $\gen{x}, \mu_{H,c}(x), c_H(x)$ from \Cref{subsn:notation} to elements $x \in \Comp(G)$.

For point (iv), see \autocite[Corollary 3.5]{geomringcomp}.
This result is central in the proofs of later results, notably \Cref{thm:galois-permuting}.
The case where $\gbar_1$ and $\gbar_3$ are both equal to the $0$-tuple is used frequently: if $x \in \Comp(G)$ and $\gamma \in \gen{x}$, then $x = x^{\gamma}$.

\subsubsection{Algebraic $k$-$G$-covers and Riemann's existence theorem.}
\label{subsubsn:algcovers}

Let $k$ be a field of characteristic relatively prime to $\card{G}$.
An \emph{algebraic cover} of $\P^1_k$ is a finite flat generically étale morphism from a smooth projective curve $Y$ over $k$ to $\P^1_k$.
Moreover, we exclude from this definition algebraic covers ramified at the point at infinity.
The $\bar k$-points of $\P^1$ over which an algebraic cover of $\P^1_k$ is ramified (i.e. the corresponding extension of local rings is ramified) form a finite configuration $\tbar \in \Conf_n(k)$, for some $n$.

A \emph{$k$-$G$-cover} is an algebraic cover $p:Y \to \P^1_k$ equipped with a group morphism from $G$ to the group $\Aut_{\P^1_k}(Y)$ of $k$-automorphisms of the cover, such that the induced action of $G$ on every unramified geometric fiber is free and transitive.

If $p:Y \rightarrow \P^1_k$ is a $k$-$G$-cover with $Y$ irreducible, then the induced extension of function fields $k(Y) \mid k(T)$ is Galois of group $G$.
If $Y$ is geometrically irreducible, this extension is \emph{regular}, i.e. $k(Y) \cap \bar k = k$.
Classically, this defines an equivalence of categories between Galois field extensions of $k(T)$ of group $G$ and $k$-$G$-covers.

Riemann's existence theorem implies that the category of $\C$-$G$-covers (or $\Qbar$-$G$-covers) branched at some configuration $\tbar \in \Conf_n(\C)$ (resp. $\tbar \in \Conf_n(\Qbar)$) is equivalent to that of topological $G$-covers with the same branch points.
We identify topological $G$-covers and $k$-$G$-covers over an algebraically closed field $k$ of characteristic zero quite freely.

Assume now $k$ is contained in $\Qbar$.
Since topological $G$-covers are well-understood, we look for \emph{$G$-covers defined over $k$} instead of $k$-$G$-covers: a $\Qbar$-$G$-cover $p : Y \rightarrow \P^1_{\Qbar}$ is \emph{defined over $k$} if it is isomorphic to the extension of scalars of a $k$-$G$-cover $p_k : Y' \to \P^1_k$.
In that case, we say that $p_k$ is a \emph{$k$-model} of the $G$-cover $p$.

\begin{remark}
  When $k$ is not algebraically closed, we avoid the ambiguous expression ``marked $k$-$G$-cover'': it is not clear whether the cover is equipped with a geometric marked point above $\infty$ (these are the $k$-$G$-covers classified by the étale fundamental group of $\P^1_k\setminus\tbar$ based at $\infty$), or with a marked $k$-point above $\infty$ (these are the marked $G$-covers invariant under the action of $\Gal(\Qbar \mid k)$).
  We mostly use the latter notion, and we refer to these covers as \emph{$k$-$G$-covers equipped with a marked $k$-point}, meaning that the marked point is in the fiber above the point at infinity.

  When $k$ is algebraically closed, we do use the terminology \emph{marked $k$-$G$-cover} since a marked geometric point above $\infty$ is also a $k$-point.

  We say that a marked $\Qbar$-$G$-cover $(p,\star)$ is \emph{defined over $k$} when it is invariant under the action of $\Gal(\Qbar\mid k)$, i.e. when it is isomorphic, as a marked $\Qbar$-$G$-cover, to the extension of scalars of $k$-$G$-cover equipped with a marked $k$-point.
  This will be defined again differently in \Cref{subsubsn:fielddef-cover}.
\end{remark}

\subsubsection{Hurwitz schemes.}
\label{subsubsn:hur-sch}

We denote by $\Hn^*(G,n)$ the Hurwitz scheme of marked $G$-covers of $\P^1$ with $n$ branch points, unramified at $\infty$.
Via the branch point morphism, the $\Q$-scheme $\Hn^*(G,n)$ is an étale cover of $\Conf_n$.
The $K$-points of $\Hn^*(G,n)$ correspond to algebraic $K$-$G$-covers branched at some configuration $\tbar \in \Conf_n(K)$, equipped with a marked $K$-point.
The existence of this moduli space follows from the same arguments as \autocite[Theorem 4.11]{rw}.
This is always a fine moduli space, unlike Hurwitz schemes of unmarked covers.
The set of $\C$-points of $\Hn^*(G,n)$, equipped with the analytic topology, is homeomorphic to $\Hur^*(G,n)$.
The geometrically irreducible components of $\Hn^*(G,n)$ are in bijection with the connected components of $\Hur^*(G,n)$, and consequently with $\B_n$-orbits of product-one $n$-tuples of elements of $G$.

The more usual Hurwitz moduli scheme of branched unmarked $G$-covers of $\P^1$ is denoted by $\Hn(G,n)$.
We use it exclusively in the discussion of \Cref{subsubsn:marked-unmarked} which relates the fields of definition of components of $\Hn^*(G,n)$ to the more classical question of the fields of definition of components of geometrically connected unmarked covers.
One thing to know is that contrary to $\Hn^*(G,n)$, the Hurwitz scheme $\Hn(G,n)$ is a \emph{coarse} moduli scheme in general: its $K$-points do not correspond precisely to $K$-$G$-covers with $n$ branch points.
These issues are detailed in \autocite{modulidef}.

\subsection{The Galois action}
\label{subsn:galact}

In this subsection, we describe the Galois action of $\Gamma_K = \Gal(\Qbar \mid K)$ on the $\Qbar$-points of the scheme $\Hn^*(G,n)$ and on its geometrically irreducible components.

We fix a configuration $\tbar \in \Conf_n(K)$, defined over $K$.
We denote by $\pi_{1,\Qbar}$ the étale fundamental group $\piet(\P^1_{\Qbar} \setminus \tbar, \infty)$ and by $\pi_{1,K}$ the étale fundamental group $\piet(\P^1_K \setminus\tbar, \infty)$.

The group $\pi_{1,\Qbar}$ is isomorphic to the profinite completion of $\pi_1(\PC\setminus \tbar, \infty)$.
Since $G$ is finite, there is a bijection between morphisms $\pi_{1,\Qbar} \rightarrow G$ and morphisms $\pi_1(\PC \setminus \tbar, \infty) \rightarrow G$.
Hence, isomorphism classes of marked $G$-covers branched at $\tbar$ may and will be seen as morphisms $\pi_{1,\Qbar} \rightarrow G$.

\subsubsection{The Galois action on covers.}
Let $\overline{\Q(T)}$ be an algebraic closure of $\Q(T)$ containing $\bar\Q$.
Let $\Omega_{\tbar}$ be the maximal subfield of $\overline{\Q(T)}$ unramified outside of $\tbar$.
We have the chain of field extensions:
\[
  \begin{tikzcd}[ampersand replacement=\&]
    {\Omega_{\tbar}} \\
    {\Qbar(T)} \\
    {K(T)}
    \arrow["{\pi_{1,\Qbar}}", no head, from=2-1, to=1-1]
    \arrow["\Gamma_K", no head, from=3-1, to=2-1]
    \arrow["{\pi_{1,K}}"', curve={height=24pt}, no head, from=3-1, to=1-1]
  \end{tikzcd}
\]
which induces the following short exact sequence:
\begin{equation}
  \label{eq:short-ex-seq}
  \begin{tikzcd}[ampersand replacement=\&]
    1 \& {\pi_{1,\Qbar}} \& {\pi_{1,K}} \& \Gamma_K \& 1
    \arrow[from=1-1, to=1-2]
    \arrow[from=1-2, to=1-3]
    \arrow[from=1-3, to=1-4]
    \arrow[from=1-4, to=1-5]
  \end{tikzcd}
  .
\end{equation}

The field $\Omega_{\tbar}$ embeds in $\overline{\Q(T)}$ which itself embeds in the field of Puiseux series over $\Qbar$, denoted by $\Qbar (((1/T)^{1/\infty}))$.
The Galois group $\Gamma_K = \Gal(\Qbar\mid K)$ acts on $\Qbar (((1/T)^{1/\infty}))$ by acting on the coefficients.
Let $\sigma \in \Gamma_K$.
Since the configuration $\tbar$ is defined over $K$, the image under $\sigma$ of an extension unramified outside of $\tbar$ is still unramified outside of $\tbar$.
This implies that the field $\Omega_{\tbar}$ is stable under the action of $\Gamma_K$.
So there is an action of $\Gamma_K$ on $\Omega_{\tbar}$, trivial on $K(T)$.
This defines a morphism:
  \[ s : \Gamma_K \rightarrow \Gal \left( \Omega_{\tbar} \mid K(T) \right) \simeq \pi_{1,K} \]
associated with the choice of the basepoint (here, the point at infinity).
The morphism $s$ is a section of the short exact sequence of \Cref{eq:short-ex-seq}:
  \[
    \begin{tikzcd}[ampersand replacement=\&]
      1 \& {\pi_{1,\Qbar}} \& {\pi_{1,K}} \& \Gamma_K \& 1.
      \arrow[from=1-1, to=1-2]
      \arrow[from=1-2, to=1-3]
      \arrow[from=1-3, to=1-4]
      \arrow[from=1-4, to=1-5]
      \arrow["s"{description}, curve={height=-12pt}, from=1-4, to=1-3]
    \end{tikzcd}
  \]
Using the morphism $s$, we define an action of $\Gamma_K$ on $\pi_{1,\Qbar}$.
For all $\sigma \in \Gamma_K$ and $\gamma \in \pi_{1,\Qbar}$, we let:
  \[
    \sigma.\gamma
    \eqdef
    \gamma^{s(\sigma)}
  \]
which belongs to $\pi_{1,\Qbar}$ by normality.
The action of an automorphism $\sigma \in \Gamma_K$ on a marked $G$-cover seen as a morphism $\varphi : \pi_{1,\Qbar} \rightarrow G$ is defined by the following equality for all $\gamma \in \pi_{1, \Qbar}$:
  \[
    (\sigma.\varphi)(\gamma)
    = \varphi(\sigma.\gamma)
    = \varphi \left( \gamma^{s(\sigma)} \right).
  \]
This action does not change the monodromy group of a marked cover branched at $\tbar$.
It satisfies $\sigma.(\varphi^g)=(\sigma.\varphi)^g$ for all $g \in G$, and so it induces a well-defined action of $\Gamma_K$ on isomorphism classes of unmarked $G$-covers branched at $\tbar$.

\subsubsection{Fields of definition of covers.}
\label{subsubsn:fielddef-cover}

Consider an isomorphism class of marked branched $G$-covers, seen as a morphism $\varphi : \pi_{1,\Qbar} \rightarrow G$.

\begin{definition}
  \label{def:fielddef}
  The marked $G$-cover associated to $\varphi$ is \emph{defined over $K$} if $\sigma.\varphi = \varphi$ for all $\sigma \in \Gamma_K$.
\end{definition}



The equivalence with the definition given in \Cref{subsubsn:algcovers} (marked $G$-covers defined over $K$ are obtained by extension of scalars of $K$-$G$-covers equipped with a marked $K$-point) follows from the properties of the étale fundamental group and from the following proposition\footnote{
  In some sense, this is also an instance of \Cref{prop:galois-descent}.
}:

\begin{proposition}
  \label{prop:defK-factorize}
  The marked cover associated to $\varphi : \pi_{1,\Qbar} \to G$ is defined over $K$ if and only if the morphism $\varphi$ has an extension to $\pi_{1,K}$ which is trivial on $\Ima(s)$, i.e. there exists a group morphism $\tilde \varphi : \pi_{1,K} \rightarrow G$ such that the following diagram commutes:
    \[\begin{tikzcd}[ampersand replacement=\&]
      {\pi_{1,\Qbar}} \& {} \\
      {\pi_{1,K}} \& G \\
      \Gamma_K \& 1
      \arrow["s", from=3-1, to=2-1]
      \arrow["{\tilde \varphi}"', from=2-1, to=2-2]
      \arrow[from=3-2, to=2-2]
      \arrow[from=3-1, to=3-2]
      \arrow[from=1-1, to=2-1]
      \arrow["\varphi", from=1-1, to=2-2]
    \end{tikzcd}\]
\end{proposition}

The triviality of $\tilde \varphi$ on $\Ima(s)$ corresponds to the $K$-rationality of the marked point above $\infty$.

\begin{proof}
  Give names to the morphisms in the exact sequence \Cref{eq:short-ex-seq}:
    \[ \pi_{1,\Qbar} \overset{\iota}{\hookrightarrow} \pi_{1,K} \overset{w}{\twoheadrightarrow} \Gamma_K \]
  and remember that  $s : \Gamma_K \rightarrow \pi_{1,K}$ is a section of $w$, i.e. $w \circ s = \id_{\pi_{1,K}}$.
  
  \begin{itemize}
  \item[($\Leftarrow$)]
    Assume there is a morphism $\tilde \varphi : \pi_{1, K} \rightarrow G$ such that $\varphi = \tilde \varphi \circ \iota$ and $\tilde\varphi \circ s = 1$.
    For $x \in \pi_{1,\Qbar}$ and $\sigma \in \Gamma_K$, compute:
      \[
        \sigma.\varphi(x) =
        \varphi \left (
          x^{s(\sigma)}
        \right ) =
        \tilde \varphi \left (
          x^{s(\sigma)}
        \right ) =
        \tilde \varphi(x)^{\tilde \varphi (s(\sigma))} =
        \tilde \varphi (x) =
        \varphi(x).
      \]
  
  \item[($\Rightarrow$)]
    Assume $\varphi$ is defined over $K$.
    Let $x \in \pi_{1,K}$. We have:
      \[ w \left (x s(w(x))^{-1} \right ) = w(x) w(s(w(x))^{-1} = w(x)w(x)^{-1} = 1, \]
    which implies that $x s(w(x))^{-1} \in \pi_{1,\Qbar}$ by exactness of \Cref{eq:short-ex-seq}.
    Define the map:
      \[
        \tilde \varphi :
        \left \lbrace
        \begin{matrix}
          \pi_{1,K} & \to & G \\
          x & \mapsto & \varphi(x s(w(x))^{-1})
        \end{matrix}
        \right .
        .
      \]
    If $x\in\pi_{1,\Qbar}$, then $w(x)=1$ and thus $\tilde\varphi(x)=\varphi(x)$.
    If $\sigma \in \Gamma_K$, then $s(w(s(\sigma))) = s(\sigma)$ so $\tilde\varphi(s(\sigma))=\varphi(s(\sigma) s(\sigma)^{-1})=\varphi(1) = 1$.
    So $\varphi = \tilde \varphi \circ \iota$ and $\tilde \varphi \circ s = 1$.
    It remains to check that $\tilde \varphi$ is a morphism.
    Let $x,y \in \pi_{1,K}$ and compute:
    \begin{align*}
      \tilde\varphi(x) \tilde\varphi(y)
      & =
      \varphi(xs(w(x))^{-1})
      \varphi(ys(w(y))^{-1})
      & \text{by definition of } \tilde\varphi
      \\ & =
      \varphi(xs(w(x))^{-1})
      (w(x).\varphi)(ys(w(y))^{-1}) &
      \text{because $\varphi$ is defined over $K$}
      \\ & =
      \varphi(xs(w(x))^{-1})\varphi \left (s(w(x))ys(w(y))^{-1}s(w(x))^{-1} \right ) &
      \text{by definition of the $\Gamma_K$-action}
      \\ & =
      \varphi(xy s(w(xy))^{-1}) =
      \tilde\varphi(xy)
      & \text{by definition of } \tilde\varphi.
    \end{align*}
    This concludes the proof.
  \end{itemize}
\end{proof}

\subsubsection{The Galois action on components and fields of definition.}
The Galois action on marked $G$-covers induces a well-defined $\Gamma_K$-action on the graded set $\Comp(G)$.
Specifically, if $m \in \Comp(G)$, an automorphism $\sigma \in \Gamma_K$ maps marked $G$-covers in the component $m$ to marked $G$-covers in the component $\sigma.m$.
We do not claim that the Galois action is compatible with the monoid structure of $\Comp(G)$: this is precisely the difficulty of \Cref{qn:prod-defined}.

\begin{definition}
  \label{def:fielddefcomp}
  A component $m \in \Comp(G)$ is \emph{defined over $K$} if for all $\sigma \in \Gamma_K$ we have $\sigma.m=m$.
\end{definition}

\subsubsection{Comparison between the marked and the unmarked cases.}
\label{subsubsn:marked-unmarked}

Let $m \in \Comp(G)$ be a component of $\Hn^*(G,n)_{\Qbar}$, and $\tilde m$ be the component of $\Hn(G,n)_{\Qbar}$ obtained by forgetting the marked points.
The component $m$ is defined over $K$ when $\sigma.m=m$ for all $\sigma \in \Gamma_K$.
A weaker property, in general, is that $\tilde m$ is defined over $K$, meaning that for all $\sigma \in \Gamma_K$ there is a $\gamma \in G$ such that $\sigma.m = m^{\gamma}$.
There are two situations to consider:
\begin{itemize}
  \item
    If $\gen{m}=G$, then, by \Cref{propitem:braid-conjugate}, we have $m^{\gamma} = m$ for all $\gamma \in G$.
    In this instance, there is no difference between $m$ and $\tilde m$ being defined over $K$.
  \item
    If $\gen{m}$ is a proper subgroup $H$ of $G$, we introduce the component $m_H$ of $\Hn^*(H,n)_{\Qbar}$ obtained by removing all connected components except for the component of the marked points from the covers in $m$, making them connected $H$-covers.
    We also introduce the component $\tilde m_H$ of $\Hn(H,n)_{\Qbar}$, obtained by forgetting the marked points of the connected $H$-covers in $m_H$.

    The fields of definition of $m_H$ and $\tilde{m}_H$ are the same, according to the previous point.
    Since the ambient group is not relevant in the definition of $\sigma.m$, it is also clear that $m$ is defined over $K$ if and only if $m_H$ is defined over $K$.

    The comparison between the situations is summarized in the figure below:
    \begin{figure}[H]
      \begin{center}
      \includegraphics[scale=0.6]{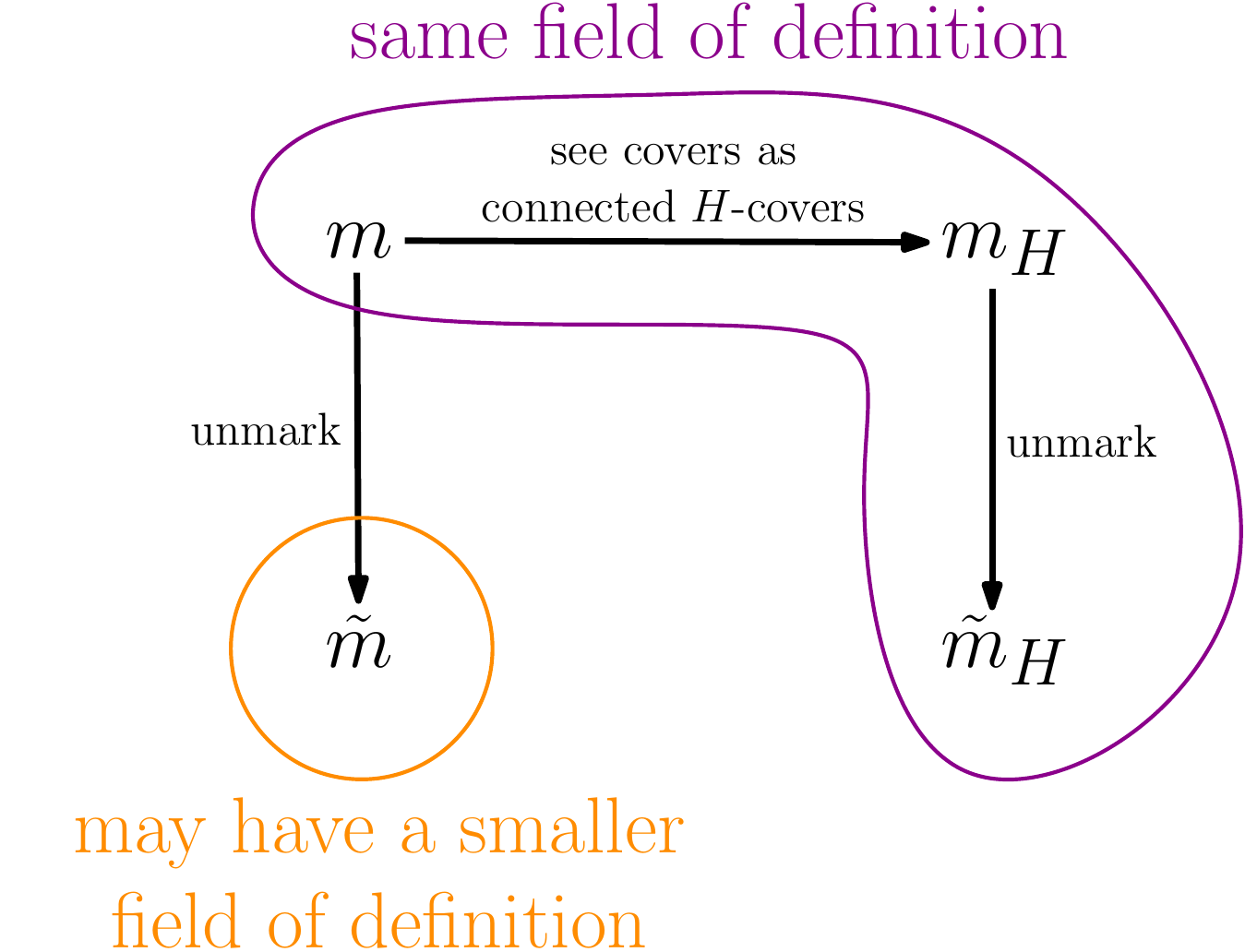}
      \caption{A summary of the situation}
      \end{center}
    \end{figure}

    This discussion implies that considering the fields of definitions of components of marked $G$-covers with monodromy group $H$ is equivalent to studying the fields of definitions of components of connected unmarked $H$-covers, which is the choice made in \autocite{Cau}.
    However, we opt for the former approach because it allows for a unified treatment of these components and leads to a simpler algebraic structure.

    There are still ways to relate the fields of definition of $\tilde{m}$ and $m$. For instance, the following lemma can be applied:
    \begin{lemma}
      If $\tilde m$ is defined over $K$ and $H$ is either self-normalizing in $G$ or has no outer automorphisms, then $m$ is defined over $K$.
    \end{lemma}
    \begin{proof}
      Consider some $\sigma \in \Gamma_K$.
      The equality $\sigma.m = m^{\gamma}$ implies $H=H^{\gamma}$, i.e. conjugation by $\gamma$ defines an automorphism of $H$.
      \begin{itemize}
        \item
          If $H$ is self-normalizing, this implies $\gamma \in H$.
        \item
          If $H$ has no outer automorphisms, conjugation by $\gamma$ has to be an inner automorphism.
          So there is a $\gamma' \in H$ such that $h^{\gamma} = h^{\gamma'}$ for all $h \in H$, and in particular $m^{\gamma} = m^{\gamma'}$.
      \end{itemize}
      In both cases, \Cref{propitem:braid-conjugate} shows that $\sigma.m = m$.
      Therefore $m$ is defined over $K$.
    \end{proof}
\end{itemize}

\subsection{The branch cycle lemma}
\label{subsn:bcl}

The action of the Galois group $\Gamma_K$ on multidiscriminants of $G$-covers or their components is precisely known.
If $\sigma \in \Gamma_K$ is an automorphism, we denote by $\sigma$ too the permutation of $\{1,\ldots,n\}$ such that $\sigma.t_i = t_{\sigma(i)}$.
Consider a marked $G$-cover branched at $\tbar$, seen as a morphism $\varphi : \pi_{1,\Qbar} \rightarrow G$.
Let $(\gamma_1,\ldots,\gamma_n) \in \pi_{1,\Qbar}^n$ be a bouquet associated to $\tbar$.
The following result is classical, cf. \autocite{Fried77} (or \autocite[Lemme 2.2]{Cau} for a statement closer to ours):

\begin{lemma}
  \label{lem:branch-cycle-lemma}
  For every $\sigma \in \Gamma_K$, the element $(\sigma.\varphi)(\gamma_i)$ is conjugate to $\left ( \varphi(\gamma_{\sigma^{-1}(i)}) \right )^{\chi(\sigma)^{-1}}$.
\end{lemma}

We restate \Cref{lem:branch-cycle-lemma} in terms of multidiscriminants of components.
Let $\gbar = (\varphi(\gamma_1),\ldots,\varphi(\gamma_n))$ be the tuple associated to $\varphi$, and $\sigma.\gbar = (\varphi(\sigma.\gamma_1),\ldots,\varphi(\sigma.\gamma_n))$ be the tuple associated to $\sigma.\varphi$ for all $\sigma \in \Gamma_K$.
Let $H$ be a group containing $\gen{\gbar}$ and $c$ be a $K$-rational (cf. \Cref{defn:rational-subset}) conjugation-invariant subset of $H$ containing $c_H(\gbar)$.
Denote by $D$ the set of conjugacy classes of $H$ contained in $c$ and by $p_{\sigma}$ the map $D \to D$ induced by the $\chi(\sigma)$-th power operation for $\sigma \in \Gamma_K$.
Let $x \in \Comp(G)$ be the component represented by the tuple $\gbar$.
Recall from \Cref{subsn:notation} that the $(H,c)$-multidiscriminant $\mu_{H,c}(x)$ of $x$ is the map that counts the occurrences in $\gbar$ of each conjugacy class in $D$.

\begin{definition}
  \label{defn:rational-multidisc}
  We say that $x$ \emph{has a $K$-rational $(H,c)$-multidiscriminant} if for all $\sigma \in \Gamma_K$ we have:
    \[ \mu_{H,c}(x) = \mu_{H,c}(x) \circ p_{\sigma}, \]
  i.e. every conjugacy class $\gamma \in D$ appears as many times in $\gbar$ as the classes $\gamma^{\chi(\sigma)}$ for $\sigma \in \Gamma_K$.
\end{definition}

Note that a product of components with $K$-rational multidiscriminants has a $K$-rational multidiscriminant.
\Cref{lem:branch-cycle-lemma} has the following corollary:

\begin{corollary}
  \label{cor:bcl}
  ~
  \begin{corenum}
    \item
    \label{coritem:galact-multidisc}
      For all $\sigma \in \Gamma_K$, the $(H,c)$-multidiscriminants of $x$ and $\sigma.x$ are related by the equality:
        \[ \mu_{H,c}(\sigma.x) = \mu_{H,c}(x) \circ p_{\sigma}. \]
    \item
    \label{coritem:bcl-rational-multidisc}
      If $x$ is defined over $K$, then $x$ has a $K$-rational $(H,c)$-multidiscriminant.
    \item
    \label{coritem:bcl-abelian}
      If $x$ has a $K$-rational $(H,c)$-multidiscriminant and $H$ is abelian, then $x$ is defined over $K$.
  \end{corenum}
\end{corollary}

Point (ii) furnishes an easily checked necessary condition for a component to be defined over $K$.
Point (iii) says that the implication of point (ii) is an equivalence in the abelian case: in this situation, fields of definition of components are well-understood.

\begin{proof}
  \begin{enumerate}[label=(\roman*)]
    \item
      Let $\gamma \in D$ be a conjugacy class of $H$.
      Then $\mu_{H,c}(\sigma.x)(\gamma)$ is the number of occurrences of $\gamma$ in $\sigma.\gbar$.
      \Cref{lem:branch-cycle-lemma} implies that this is also the number of occurrences of $\gamma^{\chi(\sigma)}$ in $\gbar$, which is precisely $\mu_{H,c}(x)(\gamma^{\chi(\sigma)})$.
    \item
      Since $x$ is defined over $K$, we have $\sigma.x = x$ for all automorphisms $\sigma \in \Gamma_K$.
      By (i), this implies $\mu_{H,c}(x) = \mu_{H,c}(x) \circ p_{\sigma}$, i.e. $x$ has a $K$-rational $(H,c)$-multidiscriminant.
    \item
      Since $H$ is abelian, conjugacy classes of $H$ and elements of $H$ are ``the same'', and components are just unordered product-one tuples of elements of $H$ (the braid group acts by permutation).
      Hence, two components are equal exactly when their $(H,c)$-multidiscriminants are equal.

      Since $x$ has a $K$-rational $(H,c)$-multidiscriminant, we know that for every $\sigma \in \Gamma_K$ we have $\mu_{H,c}(x) = \mu_{H,c}(x) \circ p_{\sigma}$, which is equal to $\mu_{H,c}(\sigma.x)$ by point (i).
      In other words, $x$ and $\sigma.x$ are components with equal $(H,c)$-multidiscriminants, and therefore they are equal.
  \end{enumerate}
\end{proof}

\begin{example}
  \label{ex:bcl-abelian}
  Assume that the group $G$ is abelian.
  \Cref{coritem:bcl-abelian} lets one determine the field of definition of components.
  A consequence is that the answer to \Cref{qn:prod-defined} is always ``yes'' in this case.
  For example, a component represented by a tuple $\gbar \in G^n$ is defined over $\Q$ exactly when every element $g \in G$ appears as many times in $\gbar$ as the elements $g^k$ for $k$ relatively prime with $\ord(g)$.
  Let us give concrete examples:
  \begin{itemize}
    \item
      The component $(1,1,1) \in \Comp(\Z/3\Z)$ is not defined over $\Q$, because $1$ does not appear as many times as $-1$.
    \item
      The component $(1,-1) \in \Comp(\Z/n\Z)$ is defined over $\Q$ for $n \in \{2,3,4,6\}$, and not defined over $\Q$ for $n=5$ or $n\geq 7$.
  \end{itemize}
\end{example}

\begin{example}
  \label{ex:symgp}
  It follows from \autocite[Theorem 10.6]{geomringcomp} that components of $\Sym_d$-covers whose monodromy elements are transpositions are entirely determined by their monodromy group $H$ and their $H$-multidiscriminant.
  Since transpositions are involutions, all conjugacy classes involved are $\Q$-rational, therefore \Cref{lem:branch-cycle-lemma} implies that the action of $\Gal(\Qbar\mid\Q)$ preserves multidiscriminants.
  Since components of $\Sym_d$-covers whose monodromy elements are transpositions are determined by their multidiscriminants, which are $\Q$-rational, these components are all defined over $\Q$.
\end{example}

\section{The group-theoretic approach}
  \label{sn:cau-thm}

In this section, we propose new applications of ideas introduced in \autocite{Cau}, which we recall in \Cref{subsn:cauthm}.
In \Cref{subsn:permuting}, we prove the main result \Cref{thmitem:prod-permuting-defK}, which corresponds to \Cref{thmitem:main-prod-permuting-defK}.
This theorem is generalized in \Cref{subsn:generalized-permuting} and examples are given in \Cref{subsn:applications-permuting}.
In \Cref{susbn:galois-small}, we use similar methods to reduce the Galois action to components of small degree, cf. \Cref{prop:reduction-galact,prop:small-deg-all-defK}.
Our approach is based on braid manipulations and group-theoretic criteria.

\subsection{Cau's theorem}
\label{subsn:cauthm}

Following \autocite{Cau}, if $x_1,\ldots,x_n \in \Comp(G)$ and $H$ is a subgroup of $G$ which contains $\gen{x_1,\ldots,x_n}$, we define the following subset of $\Comp(G)$:
  \[ \ni_H(x_1,\ldots,x_n) =
    \left \lbrace
      \prod_{i=1}^n x_i^{\gamma_i}
      \verti
      \gamma_i \in H
    \right \rbrace.
  \]

We also introduce the following subset, which always contains $x_1 \cdots x_n$:
\[ \nid_H(x_1,\ldots,x_n) =
  \left \lbrace
    \prod_{i=1}^n x_i^{\gamma_i}
    \verti
    \begin{matrix}
      \gamma_i \in H \\
      \gen{x_1^{\gamma_1}\cdots x_n^{\gamma_n}} = \gen{x_1 \cdots x_n}
    \end{matrix}
  \right \rbrace
  .
\]

When $H$ is not specified, it is assumed that $H = \gen{x_1 \cdots x_n}$.

In Cau's terminology, a family of elements of $\Comp(G)$ corresponds to a \emph{degenerescence structure} $\Delta$, and elements of $\ni_H(\Delta)$ are called \emph{$\Delta$-components}.
Cau gave a criterion to identify whether a given component is a $\Delta$-component depending on the existence of a specific ``$\Delta$-admissible cover'' on its boundary.
This characterization is key for his proof of the following theorem, which is \autocite[Théorème 3.2]{Cau}:
\begin{theorem}
  \label{thm:cau}
  Let $x_1,\ldots,x_n$ be components, $H$ a subgroup of $G$ which contains $\gen{x_1,\ldots,x_n}$, and $\sigma \in \Gamma_K$.
  Then the action of $\sigma$ on components induces a bijection:
    \[ \ni_H(x_1,\ldots,x_n) \to \ni_H(\sigma.x_1,\ldots,\sigma.x_n) \]
  and the same statement holds if $\ni_H$ is replaced by $\nid_H$.
\end{theorem}

That \Cref{thm:cau} holds if $\ni_H$ is replaced with $\nid_H$ follows from the fact that the Galois action preserves the monodromy group.
If $X$ is a finite set of components and $\sigma \in \Gamma_K$, we write \Cref{thm:cau} under the form $\sigma.\ni(X)=\ni(\sigma.X)$, where $\sigma.X$ is a shorthand for $\{\sigma.x \,\mid\, x \in X\}$.

\subsection{Permuting components}
\label{subsn:permuting}

In \autocite[Proposition 2.10]{Cau} and \autocite[Théorème 3.8]{Cau2}, Cau applies \Cref{thm:cau} in situations where he shows $\ni(x_1,\ldots,x_n) = \{x_1 \cdots x_n\}$.
We introduce a different condition that will later be shown to imply $\nid(x_1,\ldots,x_n) = \{ x_1 \cdots x_n \}$:

\begin{definition}
  \label{defn:permuting}
  Two components $x,y \in \Comp(G)$ of respective monodromy groups $H_1 = \gen{x}$ and $H_2 = \gen{y}$ are \emph{permuting} if $H_1 H_2 = \gen{H_1, H_2}$.
\end{definition}

This terminology comes from the fact that subgroups $H_1, H_2$ of a group $G$ are classically called \emph{permuting} when $H_1 H_2$ is a subgroup of $G$.
Two elements of $\Comp(G)$ are permuting exactly when their monodromy groups are permuting subgroups of $G$.
This condition is neither stronger nor weaker than the completeness conditions considered by Cau.
In \Cref{subsn:generalized-permuting}, we give a condition that generalizes both \Cref{defn:permuting} and the hypothesis of \autocite[Proposition 2.10]{Cau}.

Note that $x, y \in \Comp(G)$ are permuting whenever $\gen{x}$ or $\gen{y}$ is normal in $\gen{x,y}$, and in particular when one monodromy group contains the other.
Cases of interest are $\gen{x}=\gen{y}$ as well as $\gen{x}=G$ or $\gen{y}=G$.
Moreover, if $x,y$ are permuting and $\sigma \in \Gamma_K$, then $\sigma.x$ and $\sigma.y$ are permuting.

We are now ready to prove \Cref{thm:galois-permuting}, whose third point is \Cref{thmitem:main-prod-permuting-defK}:

\begin{theorem}
  \label{thm:galois-permuting}
  Let $x,y \in \Comp(G)$ be permuting components.
  Then:
  \begin{thmenum}
    \item
      \label{thmitem:permuting-nid-singleton}
      The set $\nid(x,y)$ contains only the component $xy$.
    \item
      \label{thmitem:sigma-prod-permuting}
      For all $\sigma\in\Gamma_K$ we have $\sigma.(xy)=(\sigma.x)(\sigma.y)$.
    \item
      \label{thmitem:prod-permuting-defK}
      If $x$ and $y$ are defined over $K$, then $xy$ is defined over $K$.
  \end{thmenum}
\end{theorem}

\begin{proof}
  \begin{enumerate}[label=(\roman*)]
    \item
      Let $\gamma, \gamma' \in H$ such that $\gen{x^{\gamma}y^{\gamma'}}=H$.
      We have to show $x^{\gamma} y^{\gamma'} = xy$.
      We use \Cref{propitem:braid-conjugate} to reduce to the case $\gamma = 1$: indeed, we have $x^{\gamma}y^{\gamma'} = xy^{\gamma^{-1}\gamma'}$ since $\gamma^{-1} \in \gen{x^{\gamma} y^{\gamma'}} = H$.
      Write $\gamma' = \gamma_1 \gamma_2$ with $\gamma_i \in H_i$.
      We have:
        \begin{align*}
          xy & = xy^{\gamma_2} & \text{by \Cref{propitem:braid-conjugate}, because } \gamma_2 \in \gen{y} \\
          & =  x(y^{\gamma_2})^{\gamma_1} & \text{by \Cref{propitem:braid-conjugate}, because } \gamma_1 \in \gen{x} \\
          & =  xy^{\gamma_1\gamma_2} \\
          & = xy^{\gamma'}.
        \end{align*}
    \item
      Let $\sigma \in \Gamma_K$.
      By \Cref{thm:cau}, the component $\sigma.(xy)$ belongs to the set $\nid(\sigma.x, \sigma.y)$ and thus it is equal to $(\sigma.x)(\sigma.y)$, by point (i) applied to the permuting components $\sigma.x$ and $\sigma.y$.
    \item
      Follows from point (ii).
  \end{enumerate}
\end{proof}

A noteworthy corollary of \Cref{thmitem:prod-permuting-defK} is the following:

\begin{corollary}
  \label{cor:power-defK}
  If $x$ is a component defined over $K$, then so is $x^n$ for all $n \geq 0$.
\end{corollary}

\begin{remark}
  To deduce \Cref{thmitem:prod-permuting-defK} from \Cref{thmitem:permuting-nid-singleton}, one can use \Cref{thm:prod-conj-defK} instead of \Cref{thm:cau}.
\end{remark}

\subsection{Applications and examples}
\label{subsn:applications-permuting}

The rest of this section is concerned with applications of \Cref{thm:galois-permuting}.

\begin{example}
  \label{ex:semi-direct-prod}
  \Cref{thm:galois-permuting} implies a slightly stronger version of \autocite[Théorème 3.5]{Cau}.
  Indeed, let $G=H_1 \ltimes H_2$ be a semi-direct product of groups.
  Assume that for $i=1,2$ there is a \emph{rigid} (cf. \autocite[Definition 2.4]{Cau}) $r_i$-tuple $\cbar_i$ of $\Q$-rational conjugacy classes of $H_i$.
  Then for $i=1,2$ there is a unique component $m_i \in \Comp(G)$ such that $\gen{m_i}=H_i$ and the $H_i$-multidiscriminant of $m_i$ counts the appearances of a class in $\cbar_i$.
  By the rigidity hypothesis and the branch cycle lemma (\Cref{cor:bcl}), these components are defined over $\Q$.
  Cau's results led him to observe that $(H_1, H_1, H_2, H_2)$ is a complete family of subgroups of $G$ and therefore $m_1^2 m_2^2$ is defined over $\Q$.
  We obtain a slightly better result: since $G$ is a semi-direct product of $H_1$ and $H_2$, the components $m_1, m_2$ are permuting and therefore the component $m_1 m_2$ is defined over $K$.

  Assume $G$ is a semi-direct product of symmetric groups: $G = \Sym_n \ltimes \Sym_m$.
  There are rigid $\Q$-rational triples of conjugacy classes of $\Sym_n$ and $\Sym_m$.
  The reasoning above shows that there is a component defined over $\Q$ of $G$-covers with six branch points, instead of the twelve needed by Cau, which already improved upon the thirty-two needed by Dèbes and Emsalem.
\end{example}

We now prove a result which can be perceived as the existence of a field norm for components:
\begin{corollary}
  \label{cor:norm-component}
  Let $x \in \Comp(G)$ be a component, and $H=\gen{x}$.
  Let $\Gamma_x$ be the subgroup of finite index of $\Gamma_K$ consisting of elements $\sigma \in \Gamma_K$ such that $\sigma.x=x$.
  Then the following component, which has monodromy group $H$, is defined over $K$:
    \[ \rmN_K(x) = \prod_{\sigma \in \Gamma_K/\Gamma_x} \sigma.x. \]
\end{corollary}

This result is a variant of \autocite[Corollaire 1.1/Corollaire 3.4]{Cau}: Cau shows that the concatenation of all components with a given degree is defined over $\Q$; here we are more precise by restricting our attention on a single Galois orbit, leading to a lower degree for the product component.

\begin{proof}
  Let $\Gamma_K.x$ be the set $\{\sigma.x \,\vert\, \sigma \in \Gamma_K/\Gamma_x\}$.
  Since all components of the form $\sigma.x$ have group $H$, repeated applications of \Cref{thmitem:permuting-nid-singleton} show that:
    \begin{equation}
      \label{eqn:ni-gamma-gbar}
      \nid_H(\Gamma_K.x) = \{\rmN_K(x)\}.
    \end{equation}
  Consider an automorphism $\sigma \in \Gamma_K$.
  The action of $\sigma$ permutes $\Gamma_K.x$.
  Finally:
  \begin{align*}
    \{ \rmN_K(x) \} 
    &= \nid_H(\Gamma_K.x) & \text{by \Cref{eqn:ni-gamma-gbar}}\\
    &= \nid_H(\sigma.(\Gamma_K.x)) & \text{because $\sigma$ permutes $\Gamma_K.x$} \\
    &= \sigma.\nid_H(\Gamma_K.x) & \text{by \Cref{thm:cau}} \\
    &= \sigma.\{\rmN_K(x)\} & \text{by \Cref{eqn:ni-gamma-gbar}} \\
    &= \{\sigma.\rmN_K(x)\}
  \end{align*}
  and thus $N_K(x)$ is defined over $K$.
\end{proof}

\subsection{Reduction of the Galois action to components of small degree}
\label{susbn:galois-small}

In this subsection we express the Galois action of $\Gamma_K$ on components in terms on the action on components of small degree.
Let $\psi(G)$ be the sum of the orders of the elements of $G$:
  \[ \psi(G) \eqdef \sum_{g \in G} \ord(g). \]

Consider a $n$-tuple $\gbar = (g_1, \ldots, g_n) \in G^n$, and let $H = \gen{\gbar}$.
If $n > \psi(G)$, then there is an element $g \in G$ which appears at least $\ord(g)+1$ times in the tuple $\gbar$.
Usual braid manipulations allow one to move these occurrences of $g$ to the beginning of the tuple.
This shows that we have the following equality in $\Comp(G)$:
  \[
    \gbar = (\underbrace{g,\ldots,g}_{\ord(g)}) y
  \]
for some $y \in \Comp(G)$ of group $H$.
Note that $(g,\ldots,g)$ and $y$ are permuting components and that $\gen{(g,\ldots,g)} = \gen{g}$ is abelian.
We have:
\begin{align*}
  \sigma.\gbar & = 
  \left (
    \sigma.(g,\ldots,g)
  \right )
  \left (
    \sigma.y
  \right ) & \text{by \Cref{thmitem:sigma-prod-permuting}} \\
  & = 
  \left
    (g^{\chi(\sigma^{-1})},\ldots,g^{\chi(\sigma^{-1})}
  \right )
  \left (
    \sigma.y
  \right )
  & \text{by the branch cycle lemma (\Cref{coritem:galact-multidisc}).}
\end{align*}

We can iterate this factorization process until the size of $y$ is smaller than $\psi(G)$: this shows that the Galois action on components is entirely determined by the cyclotomic character and the Galois action on components of small degree.
We turn this into a precise proposition:

\begin{proposition}
  \label{prop:reduction-galact}
  Let $x \in \Comp(G)$ be a component and $H = \gen{x}$.
  There are elements $g_1, \ldots, g_r \in H$ and a component $y \in \Comp(G)$ of group $H$ with $\deg(y) \leq \psi(G)$ such that:
    \[ x = \left ( \prod_{i=1}^r (\underbrace{g_i,\ldots,g_i}_{\ord(g_i)}) \right ) y. \]
  Moreover, once $x$ is expressed under this form, the Galois action of an automorphism $\sigma \in \Gamma_K$ on $x$ can be expressed in terms of the cyclotomic character $\chi$, and of the Galois action on components of degree $\leq \psi(G)$:
    \[ \sigma.x = \left ( \prod_{i=1}^r \Big( \underbrace{g_i^{\chi(\sigma^{-1})},\ldots,g_i^{\chi(\sigma^{-1})}}_{\ord(g_i)} \Big) \right ) (\sigma.y). \]
\end{proposition}

Here is another example of this phenomenon.
Let $H$ be a subgroup of $G$ and $c$ a $K$-rational conjugation-invariant subset of $H$.
Denote by $\mathcal{C}_{H,c}$ the set of components $x \in \Comp(G)$ such that $\gen{x} = H$ and $c_H(x) \subseteq c$.
Then:
\begin{proposition}
  \label{prop:small-deg-all-defK}
  Assume every component $x \in \mathcal{C}_{H,c}$ of degree $\leq 2\card{c}\psi(G)$ with a $K$-rational $(H,c)$-multidiscriminant is defined over $K$.
  Then, every component $x \in \mathcal{C}_{H,c}$ with a $K$-rational $(H,c)$-multidiscriminant is defined over $K$.
\end{proposition}

\begin{proof}
  We prove the result by induction.
  Consider a component $x \in \mathcal{C}_{H,c}$ of degree $n > 2\card{c}\psi(G)$ with a $K$-rational $(H,c)$-multidiscriminant.
  Assume that every component in $\mathcal{C}_{H,c}$ of degree $<n$ with a $K$-rational $(H,c)$-multidiscriminant is defined over $K$.
  Choose a tuple $\gbar \in c^n$ representing $x$.
  Since $n > 2\card{c}\psi(G)$, there is some $g \in c$ which appears at least $2\ord(g)\card{c} + 1$ times in $\gbar$.

  Let $g_1,\ldots,g_r$ be the elements obtained as $g^{\chi(\sigma)}$ for some $\sigma \in \Gamma_K$.
  By \Cref{coritem:bcl-abelian}, the following component, whose group is the abelian group $\gen{g}$, is defined over $K$:
    \[ y \eqdef (\underbrace{g_1,\ldots,g_1}_{\ord(g)}, \underbrace{g_2,\ldots,g_2}_{\ord(g)}, \ldots, \underbrace{g_r,\ldots,g_r}_{\ord(g)}). \]
  In particular, the component $y$ has a $K$-rational $(H,c)$-multidiscriminant by \Cref{coritem:bcl-rational-multidisc}.

  We now show that there is a component $z$ with $\gen{z} = H$ such that $x = yz$.
  For this, we apply the factorization result of \autocite[Lemma 4.6]{geomringcomp}.
  Consider a conjugacy class $\gamma$ of $H$ which appears in $y$.
  Then:
  \begin{itemize}
    \item
      The conjugacy class $\gamma$ is some $\chi(\sigma)$-th power of the conjugacy class of $g$, which appears at least $2\ord(g)\card{c}+1$ times in $\gbar$ because $g$ itself does.
      Since $x$ has a $K$-rational $(H,c)$-multidiscriminant, we have $\mu_{H,c}(x)(\gamma) \geq 2\ord(g)\card{c}+1$.
    \item
      The conjugacy class $\gamma$ appears at most $\ord(g)\card{c}$ times in $y$ since $\deg(y) \leq \ord(g)\card{c}$.
  \end{itemize}
  Finally:
  \begin{align*}
    \mu_{H,c}(x)(\gamma) & \geq 2\ord(g)\card{c}+1 \\
    & \geq \ord(g)(\card{\gamma}+\card{c}) \\
    & = \ord(\gamma)\card{\gamma} + \ord(g)\card{c} \\
    & \geq \ord(\gamma)\card{\gamma} + \mu_{H,c}(y)(\gamma).
  \end{align*}

  By \autocite[Lemma 4.6]{geomringcomp}, there exists $z \in \Comp(G)$ such that $x = yz$ and $\gen{z}=H$, and in particular $z \in \mathcal{C}_{H,c}$.
  Since $x = yz$ and $y$ both have $K$-rational $(H,c)$-multidiscriminants, the component $z$ has a $K$-rational $(H,c)$-multidiscriminant too.
  By the induction hypothesis, $z$ is defined over $K$.
  Moreover $\gen{y} \subseteq H$ so $y$ and $z$ are permuting, and thus $x = yz$ is defined over $K$ by \Cref{thmitem:prod-permuting-defK}.
  We conclude by induction.
\end{proof}

\begin{remark}
  When we have discussed the lifting invariant in \Cref{sn:liftinv}, it will appear that the hypothesis ``with a $K$-rational $(H,c)$-multidiscriminant'' in \Cref{prop:small-deg-all-defK} can be replaced by the more precise necessary condition ``whose $(H,c)$-lifting invariant is $\Gamma_K$-invariant''.
  The proof of \Cref{prop:small-deg-all-defK} can be reproduced identically except for the two following details:
  \begin{itemize}
    \item
      That $\gamma$ appears at least $2\ord(g)\card{c}+1$ times in $\gbar$ follows from the fact that a component with a $\Gamma_K$-invariant $(H,c)$-lifting invariant also has a $K$-rational $(H,c)$-multidiscriminant.
      This follows directly from the definition of the $\Gamma_K$-action on lifting invariants in \Cref{subsn:galact-liftinv}.
    \item
      To apply the induction hypothesis, we have to show that the component $z$ obtained using the factorisation lemma has a $\Gamma_K$-invariant lifting invariant.
      At that point in the proof, we know that $x = yz$, and that $x$ and $y$ both have $\Gamma_K$-invariant lifting invariants.
      First notice that $x=yz$ implies:
        \begin{equation}
          \label{eqn:pihcgbar}
          \Pi_{H,c}(x) = \Pi_{H,c}(y)\Pi_{H,c}(z).
        \end{equation}
      Now, consider an automorphism $\sigma \in \Gamma_K$.
      \Cref{thm:galois-gpact-liftinv} together with the equality $x=yz$ imply:
        \[ \sigma.\Pi_{H,c}(x) = (\sigma.\Pi_{H,c}(y))(\sigma.\Pi_{H,c}(z)) \]
      i.e.:
        \begin{equation}
          \label{eqn:pihcgbar-sigma}
          \Pi_{H,c}(x) = \Pi_{H,c}(y)(\sigma.\Pi_{H,c}(z)).
        \end{equation}
      Since the lifting invariant takes values in a group, \Cref{eqn:pihcgbar} and \Cref{eqn:pihcgbar-sigma} together imply $\Pi_{H,c}(z) = \sigma.\Pi_{H,c}(z)$.
      Hence the $(H,c)$-lifting invariant of $z$ is $\Gamma_K$-invariant.
  \end{itemize}
\end{remark}

\begin{remark}
  The constant $2 \card{c} \psi(G)$ in \Cref{prop:small-deg-all-defK} can easily be improved to:
  \[
    \sum_{\gamma \in D}
    \card{\gamma}
    \left [
      \ord(\gamma)
      \bigg(
        \card{\gamma} +
        \varphi
        \Big(
          \ord(\gamma)
        \Big)
      \bigg)
      - 1
    \right ]
  \]
  where $D$ is the set of conjugacy classes of $H$ contained in $c$ and $\varphi$ is Euler's totient function.
\end{remark}

\begin{example}
  In the situation of \Cref{ex:symgp}, where $G$ is the symmetric group $\Sym_d$ and $c$ is the set of transpositions in $G$, checking that all components of degree $\leq \frac{1}{2} d^4$ are defined over $\Q$ would have been enough to prove that they are all defined over $\Q$.
\end{example}

\subsection{Generalized permuting components}
\label{subsn:generalized-permuting}

In this subsection, we prove \Cref{thm:galois-permuting-family}, which generalizes both \Cref{thm:galois-permuting} and \autocite[Proposition 2.10]{Cau}.
First, we introduce the following definition:

\begin{definition}
  Let $x_1,\ldots,x_n \in \Comp(G)$ be components, let $H_i = \gen{x_i}$ and $H=\gen{H_1,\ldots,H_n}$.
  The family $(x_1,\ldots,x_n)$ is \emph{permuting} when for all elements $\gamma_1, \ldots,\gamma_n \in H$ and for all $i \in \{2,\ldots,n\}$ we have:
  \begin{align*}
    \text{if } & \gen{H_1,H_2,\ldots, H_{i-1}, H_i, H_{i+1}^{\gamma_{i+1}},\ldots,H_n^{\gamma_n}} = H, \\
    \text{then } & \gen{H_1,H_2,\ldots,H_{i-1}, H_{i+1}^{\gamma_{i+1}},\ldots,H_n^{\gamma_n}} H_i = H.
  \end{align*}
\end{definition}

\begin{theorem}
  \label{thm:galois-permuting-family}
  Let $(x_1,\ldots,x_n)$ be a permuting family of components.
  Then:
  \begin{thmenum}
    \item
      \label{thmitem:permuting-family-nid-singleton}
      $\nid_H(x_1,\ldots,x_n)=\{x_1 \cdots x_n\}$.
    \item
      For all automorphisms $\sigma \in \Gamma_K$, we have $\sigma.(x_1 \cdots x_n) = (\sigma.x_1) \cdots (\sigma.x_n)$.
    \item
      If $x_1,\ldots,x_n$ are defined over $K$ then $x_1 \cdots x_n$ is defined over $K$.
  \end{thmenum}
\end{theorem}

The case $n=2$ gives back \Cref{thm:galois-permuting}.
The hypothesis of \Cref{thm:galois-permuting-family} is slightly weaker than the one required to apply \Cref{thm:galois-permuting} multiple times recursively.

\begin{proof}
  We focus on proving point (i), from which points (ii) and (iii) follow like in the proof of \Cref{thm:galois-permuting}.
  Let $\gamma_1,\ldots,\gamma_n \in G$ such that $\gen{\prod x_i^{\gamma_i}}=H$.
  First we can assume $\gamma_1 = 1$ as in the proof of \Cref{thmitem:permuting-nid-singleton}.
  We proceed by induction.
  Assume we have shown:
  \[ x_1 \cdots x_n = x_1 x_2 \cdots x_{i-1} x_i x_{i+1}^{\gamma_{i+1}} \cdots x_n^{\gamma_n}. \]
  In particular, we have $\gen{H_1,H_2,\ldots, H_{i-1}, H_i, H_{i+1}^{\gamma_{i+1}},\ldots,H_n^{\gamma_n}} = H $.
  Since $(x_1,\ldots,x_n)$ is permuting, we can write $\gamma_i = \gamma_i^{(1)} \gamma_i^{(2)}$ with $\gamma_i^{(1)} \in \gen{H_1,H_2,\ldots,H_{i-1}, H_{i+1}^{\gamma_{i+1}},\ldots,H_n^{\gamma_n}}$ and $\gamma_i^{(2)} \in H_i$.
  Therefore:
    \begin{align*}
      x_1 \cdots x_n
      & = x_1 \cdots x_{i-1}x_i^{\gamma_i^{(2)}} x_{i+1}^{\gamma_{i+1}} \cdots x_n^{\gamma_n}
      & \text{by \Cref{propitem:braid-conjugate}, because } \gamma_i^{(2)} \in \gen{x_i}
      \\ & = x_1 \cdots x_{i-1} \left ( x_i^{\gamma_i^{(2)}} \right )^{\gamma_i^{(1)}} x_{i+1}^{\gamma_{i+1}} \cdots x_n^{\gamma_n} & \text{ because } \gamma_i^{(1)} \in \gen{x_1 \cdots x_{i-1}, x_{i+1}^{\gamma_{i+1}}  \cdots x_n^{\gamma_n} }
      \\ & =x_1 \cdots x_{i-1} x_i^{\gamma_i} x_{i+1}^{\gamma_{i+1}}  \cdots x_n^{\gamma_n}
    \end{align*}
  and we conclude by induction.
\end{proof}

We now give an application of \Cref{thm:galois-permuting-family}:

\begin{example}
  \label{ex:v-evw2}
  Let $c$ be a $K$-rational conjugation-invariant set of $G$.
  Assume that $c$ is \emph{complete}, i.e. no proper subgroup of $G$ intersects every conjugacy class contained in $c$ (for example, Jordan's lemma implies that $c=G\setminus\{1\}$ is complete).

  The following component (introduced in \autocite[Paragraph 5.5]{EVW2}) is defined over $K$:
    \[ V = \prod_{g \in c} (\underbrace{g,\ldots,g}_{\ord(g)}). \]

  Indeed, consider an automorphism $\sigma \in \Gamma_K$.
  Since $\gen{g}$ is abelian, \Cref{coritem:galact-multidisc} implies that $\sigma.(g,\ldots,g)=(g^{\chi(\sigma^{-1})},\ldots,g^{\chi(\sigma^{-1})})$.
  The profinite integer $\chi(\sigma^{-1})$ is invertible and so the action of $\sigma$ permutes the factors of $V$.
  Now:
    \[
      \sigma.V
      = \sigma.\prod_{g \in c} (g,\ldots,g)
      \in \nid(\left \{(g,\ldots,g) \verti g \in c \right \}).
    \]
  We want to apply \Cref{thmitem:permuting-family-nid-singleton} to show that $\nid(\left \{(g,\ldots,g) \verti g \in c \right \})$ is a singleton, from which $\sigma.V=V$ follows.
  Consider an element $g \in c$ and elements $\gamma_{g'} \in G$ for all $g' \in c\setminus\{g\}$, such that $G$ is generated by $g$ together with the elements $(g')^{\gamma_{g'}}$ for $g' \in c\setminus\{g\}$.
  We want to show $\gen{(g')^{\gamma_{g'}} \text{ for } g' \in c\setminus\{g\}} \gen{g} = G$.
  There are two distinct cases:
  \begin{itemize}
    \item
      If $g \in Z(G)$, then this follows easily from, say, the fact that $\gen{g}$ is normal in $G$.
    \item
      If $g \not\in Z(G)$, then there is a $g' \in c\setminus\{g\}$ such that $g$ and $g'$ are conjugate.
      Therefore $\gen{g'^{\gamma_{g'}} \text{ for } g' \in c\setminus\{g\}}$ is a subgroup of $G$ that intersects every conjugacy class contained in $c$, and therefore it equals $G$ by the completeness assumption.
  \end{itemize}
\end{example}

\section{The lifting invariant approach}
\label{sn:liftinv}

In this section, we use the lifting invariant of \autocite{EVW2, wood} to study \Cref{qn:prod-defined}.
We first recall known properties of this invariant (\Cref{subsn:liftinv}) and then give arithmetic applications (\Cref{subsn:applications-liftinv} and \Cref{subsn:galact-liftinv}), including \Cref{thm:mbig-prod} (which is \Cref{thmitem:main-prod-Mbig-defK}).

\subsection{Presentation of the lifting invariant}
\label{subsn:liftinv}

\subsubsection{Definition and first properties.}

In this subsection, we present the lifting invariant.
For exhaustivity and convenience of the reader, we include some proofs for known facts.
In what follows, $H$ is a subgroup of $G$ and $D$ is a set of conjugacy classes of $H$ which together generate $H$.
We denote by $c$ the union of the conjugacy classes in $D$.

We define the group $U(H,c)$ in the following way: it is generated by elements $[g]$ for each $g \in c$, satisfying the relations $[g][h][g]^{-1}=[ghg^{-1}]$ for all $g,h \in c$.

\begin{definition}
  Let $\gbar = (g_1,\ldots,g_n) \in c^n$ be a tuple.
  Its \emph{(H,c)-lifting invariant} is the following element of $U(H,c)$:
    \[ \Pi_{H,c}(\gbar) = [g_1]\cdots[g_n]. \]
\end{definition}


\begin{proposition}
  \label{prop:liftinv-braidinv}
  The $(H,c)$-lifting invariant $\Pi_{H,c}(\gbar)$ depends only on the braid orbit of $\gbar$.
\end{proposition}

\begin{proof}
  The relation $[g][h][g]^{-1}=[ghg^{-1}]$ in $U(H,c)$ can be rewritten as $[g][h]=[h^g][g]$.
  So the $(H,c)$-lifting invariant is unchanged by elementary braids, which generate the braid group.
\end{proof}

\Cref{prop:liftinv-braidinv} implies that we can talk about the $(H,c)$-lifting invariant $\Pi_{H,c}(x)$ of a component $x \in \Comp(G)$ as soon as $\gen{x} \subseteq H$ and $c_H(x) \subseteq c$.
If $(H,c)$ is not specified, the \emph{lifting invariant} of a component is its $(H,c)$-lifting invariant with $H=\gen{\gbar}$ and $c=c_H(\gbar)$.

We denote by $\pi$ the morphism $U(H,c) \rightarrow H$ induced by the formula $[g] \mapsto g$.
The use of the letter $\pi$ is justified by the observation that for all $\gbar \in c^n$, we have $\pi(\Pi_{H,c}(\gbar))=\pi\gbar$.
The kernel of the morphism $\pi$ is denoted by $U_1(H,c)$.

\begin{proposition}
  \label{prop:ker-pi-central}
  The subgroup $U_1(H,c) = \ker(\pi)$ is central in $U(H,c)$.
\end{proposition}

\begin{proof}
  Let $x$ be an element of $U_1(H,c)$ decomposed as $[g_1]^{\varepsilon_1} \cdots [g_n]^{\varepsilon_n}$  with $g_1^{\varepsilon_1} \cdots g_n^{\varepsilon_n} = 1$ and $\varepsilon_i \in \{-1,1\}$.
  Now consider one of the generators $[h]$ of $U(H,c)$.
  We have:
  \begin{align*}
    x [h] & = [g_1]^{\varepsilon_1} \cdots [g_{n-1}]^{\varepsilon_{n-1}} [g_n]^{\varepsilon_n} [h] 
    \\ & = [g_1]^{\varepsilon_1} \cdots [g_{n-1}]^{\varepsilon_{n-1}} [h^{g_n^{\varepsilon_n}}] [g_n]^{\varepsilon_n}
    \\ & = [h^{g_1^{\varepsilon_1} \cdots g_n^{\varepsilon_n}}] [g_1]^{\varepsilon_1} \cdots [g_{n-1}]^{\varepsilon_{n-1}} [g_n]^{\varepsilon_n}
    \\ & = [h] x.
  \end{align*}
  Hence $x$ commutes with a generating set of elements of $U(H,c)$ and thus $x \in Z(U(H,c))$.
  This shows that $U_1(H,c)$ is a central subgroup of $U(H,c)$.
\end{proof}

It follows from \Cref{prop:liftinv-braidinv} and \Cref{prop:ker-pi-central} that $\Pi_{H,c}$ induces a morphism of monoids from the following submonoid of $\Comp(G)$:
  \[
    \Comp(H,c) =
    \left \lbrace
      m  \in \Comp(G)
      \verti
      \begin{matrix}
        \gen{m}\subseteq H \\
        c_H(m) \subseteq c
      \end{matrix}
    \right \rbrace
  \]
into the abelian group $U_1(H,c)$.
In fact, the abelian group $U_1(H,c)$ is the Grothendieck group of $\Comp(H,c)$, i.e. $\Pi_{H,c}$ is universal among morphisms from $\Comp(H,c)$ into groups.
This follows from \autocite[Theorem 7.5.1]{EVW2}.
In some sense, this means that the lifting invariant is the best possible group-valued multiplicative invariant for components.

The $(H,c)$-multidiscriminant defines a morphism from $\Comp(H,c)$ into the group $\Z^D$ of maps from $D$ to $\Z$.
By the universal property of the Grothendieck group, the $(H,c)$-multidiscriminant can be recovered from the $(H,c)$-lifting invariant.
Indeed, the formula:
  \[ [g] \mapsto (0, 0, \ldots, 0, 1, 0, \ldots, 0), \]
where the nonzero coefficient occurs at the coordinate corresponding to the conjugacy class of $g$, induces a morphism $U(H,c) \to Z^D$ whose restriction is the expected morphism $U_1(H,c) \rightarrow \Z^D$.
So the $(H,c)$-lifting invariant refines the $(H,c)$-multidiscriminant.

The $(G,G)$-lifting invariant of a component $x \in \Comp(G)$ only depends on the component of unmarked covers obtained by forgetting the marked points of covers in $x$.
This follows from the following proposition:

\begin{proposition}
  \label{prop:liftinv-conjinv}
  If $\gamma \in H$ and $x \in \Comp(H,c)$ then $\Pi_{H,c}(x) = \Pi_{H,c}(x^{\gamma})$.
\end{proposition}

\begin{proof}
  Since $c$ generates $H$, we can choose elements $\gamma_1,\ldots,\gamma_n \in c$ such that $\gamma_1 \cdots \gamma_n = \gamma$.
  In $U(H,c)$, we have:
  \begin{align*}
    [\gamma_1]\cdots[\gamma_n]\Pi_{H,c}(x) & = [\gamma_1]\cdots[\gamma_{n-1}] \Pi_{H,c}(x^{\gamma_n}) [\gamma_n]
    \\ & = \Pi_{H,c}(x^{\gamma_1 \cdots \gamma_n}) [\gamma_1]\cdots [\gamma_n]
    \\ & = \Pi_{H,c}(x^{\gamma}) [\gamma_1]\cdots [\gamma_n].
  \end{align*}
  By \Cref{prop:ker-pi-central}, the element $\Pi_{H,c}(x^{\gamma}) \in U_1(H,c)$ is central and thus we can cancel $[\gamma_1]\cdots [\gamma_n]$ in this equality.
  This concludes the proof.
\end{proof}

\subsubsection{The structure of the group $U(H,c)$.}
\label{subsubsn:uhc-product}
We give a quick description of the group $U(H,c)$.
The proofs for the facts stated here are contained in \autocite[Paragraph 2.1]{wood}.
The main result is that the group $U(H,c)$ is isomorphic to a fibered product:
  \[ U(H,c) \simeq S_c \underset{H^{\ab}}{\times} \Z^D \]
where $S_c$ (a \emph{reduced Schur cover} of $G$) is a finite group which fits in an exact sequence:
  \[ 1 \rightarrow H_2(H,c) \rightarrow S_c \rightarrow H \rightarrow 1 \]
for a specific quotient $H_2(H,c)$ of the second homology group $H_2(H,\Z)$ of $H$.
In particular, the central subgroup $U_1(H,c)$ is isomorphic to the following direct product:
  \[ H_2(H,c) \times \ker(\tilde \pi) \]
where the morphism $\tilde \pi : \Z^D \to H^{\ab}$ is defined in the following way: when $\gamma \in D$ is a conjugacy class, denote by $\tilde \gamma$ the (well-defined) image in $H^{\ab}$ of the elements of $\gamma \in D$; if $\psi \in \Z^D$, we then let:
  \[ \tilde\pi(\psi) = \prod_{\gamma \in D} {\tilde \gamma}^{\psi(\gamma)}. \]

\subsubsection{The lifting invariant distinguishes ``big'' components.}
The following result, which is proved in \autocite[Theorem 7.6.1]{EVW2} and \autocite[Theorem 3.1]{wood}, states that components are entirely determined by their $(H,c)$-lifting invariant as soon as each conjugacy class of $H$ in $c$ is represented enough times in their $(H,c)$-multidiscriminant.
This is a stronger version of a result known as the Conway-Parker theorem.

To state the result, we introduce some notation.
If $\psi \in \Z^D$, we denote by $\card{\psi}$ the value $\sum_{\gamma \in D} \psi(\gamma)$ and by $\min(\psi)$ the minimal value taken by $\psi(\gamma)$ for $\gamma \in D$.
Then:
\begin{theorem}
  \label{thm:conwayparker}
  There is a constant $M_{H,c} \in \N$ such that for all $\psi \in \Z^D$ satisfying $\min(\psi) \geq M_{H,c}$, the morphism $\Pi_{H,c}$ induces a bijection:
  \[
    \bigslant
    {
      \left \lbrace
        \gbar \in G^{\card{\psi}}
        \verti
        \begin{matrix}
          \gen{\gbar} = H \\
          c_H(\gbar) = c \\
          \mu_{H,c}(\gbar) = \psi
        \end{matrix}
      \right \rbrace
    }
    {
      \B_{\card{\psi}}
    }
    \overset{\sim}{\longrightarrow}
    \left \lbrace
      x \in U(H,c)
      \verti
      x \text{ has image } \psi \text{ in } \Z^D
    \right \rbrace
    .
  \]
\end{theorem}

\subsection{The lifting invariant and fields of definition of glued components}
\label{subsn:applications-liftinv}

In this subsection, we use \Cref{thm:conwayparker} to prove \Cref{thmitem:main-prod-Mbig-defK}.
The proof also makes use of \Cref{thm:cau}.

First, remark that we may choose a (rough) constant $M$ independent from $(H,c)$ that satisfies the conclusion of \Cref{thm:conwayparker}:
  \[
    M \eqdef \max_{(H,c)} M_{H,c},
  \]
where the maximum is taken over couples $(H,c)$ where $H$ is a subgroup of $G$ and $c$ is a conjugation-invariant subset of $H$ that generates $H$.
In what follows, the constant $M$ is fixed in this way.

\begin{definition}
  A tuple $\gbar$ of elements of $G$ is \emph{$M$-big} if every conjugacy class of $H = \gen{\gbar}$ that appears in $\gbar$ appears at least $M$ times, i.e. $\min \mu_H(\gbar) \geq M$.
  A component $x \in \Comp(G)$ is \emph{$M$-big} if its representing tuples are $M$-big.
\end{definition}

\Cref{thm:conwayparker} implies that $M$-big components are determined by their lifting invariant.
Note that if $x \in \Comp(G)$ and $k \geq M$, the component $x^k$ is always $M$-big.
We now prove \Cref{thm:mbig-prod}, which is \Cref{thmitem:main-prod-Mbig-defK}:

\begin{theorem}
  \label{thm:mbig-prod}
  Let $x_1, \ldots, x_n \in \Comp(G)$ be components such that $x_1 \cdots x_n$ is $M$-big.
  Then:
  \begin{thmenum}
    \item
      \label{thmitem:irred-niels-mbig}
      The set $\nid(x_1,\ldots,x_n)$ contains only the component $x_1 \cdots x_n$.
    \item
      For all automorphisms $\sigma \in \Gamma_K$, we have $\sigma.(x_1 \cdots x_n) = (\sigma.x_1)\cdots(\sigma.x_n)$.
    \item
      \label{thmitem:prod-Mbig-defK}
      If $x_1, \ldots, x_n$ are defined over $K$, then $x_1 \cdots x_n$ is defined over $K$.
  \end{thmenum}
\end{theorem}

\begin{proof}
  Let $H = \gen{x_1 \cdots x_n}$ and $c = c_H(x_1 \cdots x_n)$.
  It follows from \Cref{prop:liftinv-conjinv} and from the multiplicativity of $\Pi_{H,c}$ that all elements of $\nid(x_1,\ldots,x_n)$ have the same $(H,c)$-lifting invariant.
  Moreover they are all $M$-big and have monodromy group $H$.
  By \Cref{thm:conwayparker}, they must be equal to each other.
  This proves point (i).
  Points (ii) and (iii) follow from point (i) and from \Cref{thm:cau}, like in the proof of \Cref{thm:galois-permuting}.
\end{proof}

\Cref{thmitem:prod-Mbig-defK} coupled with \Cref{cor:power-defK} imply that if $x_1,\ldots,x_n$ are defined over $K$, and if $k \geq M$, then $(x_1 \cdots x_n)^k$ is defined over $K$.

\subsection{The Galois action on lifting invariants}
\label{subsn:galact-liftinv}

Let $H$ be a subgroup of $G$, $c$ be a $K$-rational conjugation-invariant subset of $H$ which generates $H$, and $D$ be the set of conjugacy classes of $H$ contained in $c$.
In this subsection, we define a Galois action of $\Gamma_K$ on the set $U(H,c)$.
\Cref{prop:galois-liftinv} implies that this action effectively describes the effect on lifting invariants of the Galois action on $\Comp(G)$.
This generalizes the branch cycle lemma (\Cref{coritem:galact-multidisc}).
Moreover, in \Cref{thm:galois-gpact-liftinv}, we show that the Galois action on lifting invariants of elements of $\Comp(G)$ is compatible with multiplication.

Consider a Galois automorphism $\sigma \in \Gamma_K$.
Since $c$ is a $K$-rational subset, the $\chi(\sigma)$-th power operation defines a map $p_{\sigma} : D \rightarrow D$.

If $\gamma \in D$, choose an arbitrary element $g_{\gamma}$ of $\gamma$ and denote by $\hat{g_{\gamma}}$ (resp. $\hat{(g_{\gamma})^{\chi(\sigma^{-1})}}$) the projection on $S_c$ of the element $[g_{\gamma}] \in U(H,c)$ (resp. $\left [ (g_{\gamma})^{\chi(\sigma^{-1})} \right ] \in U(H,c)$).
Define the following element of $S_c$, which can be checked to be independent from the choice of $g_{\gamma}$:
  \[ w(\gamma, \sigma) \eqdef {\hat {g_\gamma}}^{-\chi(\sigma^{-1})} \widehat {(g_\gamma)^{\chi(\sigma^{-1})}}. \]
Importantly, the element $w(\gamma, \sigma)$ is central in $S_c$ (its image in $H$ is $(g_\gamma)^{-\chi(\sigma^{-1})} (g_\gamma)^{\chi(\sigma^{-1})} = 1$, and thus \Cref{prop:ker-pi-central} applies).

Consider an element $v \in U(H,c)$, decomposed as $(h,\psi)$ via the isomorphism $U(H,c) \simeq S_c \underset{H^{\ab}}{\times} \Z^D$ (cf. \Cref{subsubsn:uhc-product}).
We let:
\[
  \sigma.v
  = \left (
    \,\,
    h^{\chi(\sigma^{-1})}
    \prod_{\gamma \in D}
      w(\gamma, \sigma)^{\psi(c)}
    \hspace{.5cm},\hspace{.5cm}
    \psi \circ p_{\sigma}
    \,\,
  \right ).
\]
This formula is shown to define an action of $\Gamma_K$ on the set $U(H,c)$.
This construction is taken from \autocite[Paragraph 4.1]{wood}, and the following proposition follows from \autocite[Paragraph 6.1]{wood}:

\begin{proposition}
  \label{prop:galois-liftinv}
  Let $x \in \Comp(G)$ with $\gen{x} \subseteq H$ and $c_H(x) \subseteq c$.
  Then:
    \[
      \Pi_{H,c}(\sigma.x) = \sigma.\Pi_{H,c}(x).
    \]
\end{proposition}

By projection on $\Z^D$, \Cref{prop:galois-liftinv} gives back the branch cycle lemma (\Cref{coritem:galact-multidisc}).
A consequence of \Cref{prop:galois-liftinv} is the following necessary condition, which refines \Cref{coritem:bcl-rational-multidisc}:
\begin{corollary}
  \label{cor:liftinv-defK}
  Let $x \in \Comp(G)$ with $\gen{x} \subseteq H$ and $c_H(x) \subseteq c$.
  If the component $x$ is defined over $K$, then its $(H,c)$-lifting invariant is $\Gamma_K$-invariant.
\end{corollary}

We now show that a product of $\Gamma_K$-invariant elements of $U_1(H,c)$ is $\Gamma_K$-invariant, and thus the lifting invariant cannot be used to detect negative answers to \Cref{qn:prod-defined}.
This follows from the following fact:
\begin{theorem}
  \label{thm:galois-gpact-liftinv}
  The action of $\Gamma_K$ on $U_1(H,c)$ is compatible with multiplication.
\end{theorem}

\begin{proof}
  Let $v,v' \in U_1(H,c)$, and decompose them as $v = (h,\psi)$, $v' = (h',\psi')$ with $h,h' \in H_2(H,c)$ and $\psi,\psi' \in \ker(\tilde\pi)$.
  We have $vv' = (hh', \psi + \psi')$.
  Let $\sigma \in \Gamma_K$.
  With notation as above, we have:
  \begin{align*}
    \sigma.(vv') &
    = \left (
      \,\,
      (hh')^{\chi(\sigma^{-1})}
      \prod_{\gamma \in D}
        w(\gamma, \sigma)^{(\psi+\psi')(c)}
        \hspace{.5cm},\hspace{.5cm}
      (\psi + \psi') \circ p_{\sigma}
      \,\,
    \right ) \\ &
    = \left (
      \,\,
      h^{\chi(\sigma^{-1})}
      (h')^{\chi(\sigma^{-1})}
      \prod_{\gamma \in D}
        w(\gamma, \sigma)^{\psi(c)}
        w(\gamma, \sigma)^{\psi'(c)}
        \hspace{.5cm},\hspace{.5cm}
      \psi \circ p_{\sigma} + \psi' \circ p_{\sigma}
      \,\,
    \right ) \\
    & = \left(
      \,\,
      \Bigg(
        h^{\chi(\sigma^{-1})}
        \prod_{\gamma \in D}
        w(\gamma, \sigma)^{\psi(c)}
      \Bigg)
      \Bigg(
        (h')^{\chi(\sigma^{-1})}
        \prod_{\gamma \in D}
          w(\gamma, \sigma)^{\psi'(c)}
        \Bigg)
      \hspace{.5cm},\hspace{.5cm}
      \psi \circ p_{\sigma} + \psi' \circ p_{\sigma}
      \,\,
    \right) \\
    & = (\sigma.v) (\sigma.v').
  \end{align*}
\end{proof}

We have used repeatedly that $H_2(H,c)$ is abelian in the final computation: this proof does not apply to arbitrary elements of $U(H,c)$.
However, the same proof shows that $\sigma.(vv') = (\sigma.v)(\sigma.v')$ holds as soon as $v$ and $v'$ commute in $U(H,c)$.

\Cref{thm:galois-gpact-liftinv} implies positive answers to \Cref{qn:prod-defined} in situations where the lifting invariant is shown to characterize components.
For example, \Cref{thmitem:prod-Mbig-defK} can be deduced from \Cref{thm:conwayparker} by using \Cref{thm:galois-gpact-liftinv} instead of \Cref{thm:cau}.

\section{The patching approach}
\label{sn:patching}

In this section, we use patching results over complete valued fields to study the fields of definition of components obtained by gluing two components $x,y \in \Comp(G)$ defined over the number field $K$.
The main result is \Cref{thm:prod-conj-defK}, which is \Cref{thmitem:main-prod-conj-defK}.
We now give a sketch of the argument, which also serves as an outline of the section.

In \Cref{subsn:lindisj-hit}, we recall and use a version of Hilbert's irreducibility theorem (\Cref{thm:hilbirred}) to construct infinitely many field extensions $K_1, K_2, \ldots$ of $K$, pairwise linearly disjoint, over which the components $x$ and $y$ both have points (\Cref{lem:magic-lemma}).
For each $n \in \N$, denote by $f_n$ (resp. $g_n$) a $K_n((X))$-point of $x$ (resp. $y$) obtained from a $K_n$-point of $x$ (resp. $y$).
Note that $K_n((X))$ is a complete valued field for the $(X)$-adic valuation.

In \Cref{subsn:patch-glue}, we prove that for each $n \in \N$, the cover obtained by patching the $K_n((X))$-$G$-covers $f_n$ and $g_n$ is a $K_n((X))$-$G$-cover which lies in a component $m_n \in \nid(x,y)$ (\Cref{lem:patching-in-ni}).
In particular, the field of definition of the component $m_n$ is included in $\Qbar \cap K_n((X)) = K_n$.

Finally, we observe that at least two components $m_n, m_{n'}$ have to be equal because $\nid(x,y)$ is finite.
Such a component $m_n = m_{n'}$ has its field of definition included in $K_n \cap K_{n'} = K$.
In other words, we have found a component defined over $K$ in $\nid(x,y)$: this is precisely \Cref{thm:prod-conj-defK}.
The detailed proof is the focus of \Cref{subsn:proof-prodconj}.

Note that the results of this section rely crucially on the fact that number fields are \emph{Hilbertian}.

\subsection{Constructing covers with linearly disjoint fields of definition}
\label{subsn:lindisj-hit}

We will use the following form of Hilbert's irreducibility theorem, which is close to the statement in \autocite{bary}:

\begin{theorem}[Hilbert's irreducibility theorem]
  \label{thm:hilbirred}
  Let $L' \mid L$ be a finite extension of number fields and $p : X \rightarrow Y$ be a finite étale morphism from a variety $X$ over $L$ to an open subvariety $Y$ of $\mathbb{A}^n_L$.
  Assume $X_{L'}$ is irreducible.
  Then there exists an $L$-point $t \in Y(L)$ such that the $\Qbar$-points of $X$ that are mapped to $t$ by $p$ lie in a single $\Gal(\Qbar\mid L')$-orbit.
\end{theorem}

When $L=L'$, this theorem is well-known.
The fact that $L'$ may be chosen larger than $L$ follows from \autocite[Corollary 12.2.3]{fieldar}.
Using \Cref{thm:hilbirred}, we prove the following lemma:

\begin{lemma}
  \label{lem:magic-lemma}
  Let $L' \mid L$ be a finite Galois extension of number fields and $S$ be an irreducible component of the Hurwitz scheme $\Hn^*(G,n)_L$ which is geometrically irreducible.
  Then there exists a field extension $\tilde L \mid L$ such that $\tilde L$ and $L'$ are linearly disjoint over $L$, and an $\tilde L$-point $f \in S(\tilde L)$.
\end{lemma}

\begin{proof}
  Since $S$ is geometrically irreducible, its extension $S_{L'}$ is irreducible.
  The branch point morphism $S \rightarrow (\Conf_n)_{L}$ is finite étale.
  By Hilbert's irreducibility theorem (\Cref{thm:hilbirred}), there is a configuration $\tbar \in \Conf_n(L)$ such that the fiber $F \subset S(\Qbar)$ above $\tbar$ consists of a single Galois orbit.

  Let $f$ be any of the points in $F$ and $\tilde L$ be the smallest extension of $L$ over which the point $f$ is rational.
  Note that $\tilde L L'$ is the smallest extension of $L'$ over which the point $f$ is rational.
  The fiber $F$ is the $\Gal(\Qbar\mid L')$-orbit of $f$, hence the degree of the extension $\tilde L L' \mid L'$ is the cardinality of $F$.
  The same argument shows that the degree of the extension $\tilde L \mid L$ is also equal to the cardinality of $F$.
  The equality $[\tilde L L':L'] = [\tilde L : L]$ implies that $L'$ and $\tilde L$ are linearly disjoint over $L$.
\end{proof}

\subsection{Relating patching and gluing}
\label{subsn:patch-glue}

Let $\mathcal{O}$ be a complete discrete valuation ring of characteristic zero and $L$ its fraction field, which is a complete valued field.
Since $\Qbar$ is algebraically closed and $\bar L$ contains $\Qbar$, extension of scalars induces a bijection between the irreducible components of $\Hn^*(G,n)_{\Qbar}$ and those of $\Hn^*(G,n)_{\bar L}$.
We denote by $\Phi_n$ this bijection.
This lets us do the following slight terminological abuse: we say that a component $x \in \Comp(G)$ has an $L$-rational point if its extension $\Phi_{\deg(x)}(x)$ to $\bar L$ has an $L$-rational point.

\begin{lemma}
  \label{lem:patching-in-ni}
  Let $x_1, x_2 \in \Comp(G)$ be components which have $L$-rational points.
  Then there is a component $y \in \nid(x_1, x_2)$ which has an $L$-rational point.
\end{lemma}

\begin{proof}
  ~
  \begin{itemize}
    \item
      \textbf{Step 1: Setting things up}

      For $i=1,2$, let $r_i = \deg(x_i)$, $G_i = \gen{x_i}$ and fix an $L$-model $f_i \in \Hn^*(G, r_i)(L)$ of an $L$-rational point of $\Phi_{r_i}(x_i)$.
      The point $f_i$ corresponds to an $L$-$G$-cover with a marked $L$-point above $\infty$.
      In the cover $f_i$, keep only the geometrically irreducible component of the marked point, which is defined over $L$ since the marked point is $L$-rational.
      This turns $f_i$ into a geometrically irreducible $L$-$G_i$-cover with a marked $L$-point.
      The cover $f_i$ belongs to the component $x_i'$ of $\Hn^*(G_i, r_i)_L$ obtained by keeping only the component of the marked points in the covers of $x_i$, like in \Cref{subsubsn:marked-unmarked}.
      Without loss of generality, we may assume $G = \gen{G_1, G_2}$.

    \item
      \textbf{Step 2: Patching covers over $L$}

      We use the algebraic patching results of \autocite{haran}.
      First define:
      \[\begin{matrix}
          L\{z\}
          =
          \left \lbrace
            \sum_{i \geq 0} a_i z^i \in L[[z]]
            \verti
            a_i \underset{i \to \infty}{\to} 0
          \right \rbrace
          & \hspace{0.5cm} &
          Q_1 = \Frac(L\{z\}) \\
          L\{z^{-1}\}
          =
          \left \lbrace
            \sum_{i \geq 0} a_i z^{-i} \in L[[z^{-1}]]
            \verti
            a_i \underset{i \to \infty}{\to} 0
          \right \rbrace
          & &
          Q_2 = \Frac(L\{z^{-1}\}) \\
          L\{z, z^{-1}\}
          =
          \left \lbrace
            \sum_{i \in \Z} a_i z^i \in L[[z, z^{-1}]]
            \verti
            a_i \underset{i \to \pm\infty}{\to} 0
          \right \rbrace
          & &
          \hat Q = \Frac(L\{z, z^{-1}\}).
      \end{matrix}\]
      Let also $Q'_1 = Q_2$ and $Q'_2 = Q_1$.
      From the point of view of rigid analytic geometry, $Q_1$ (resp. $Q_2$, and $\hat Q$) is the algebra of analytic functions on the unit disk $D_1$ centered at $0$ (resp. a disk $D_2$ centered at $\infty$, and the annulus $D_1 \cap D_2$):

      \begin{figure}[H]
        \begin{center}
        \includegraphics[scale=0.5]{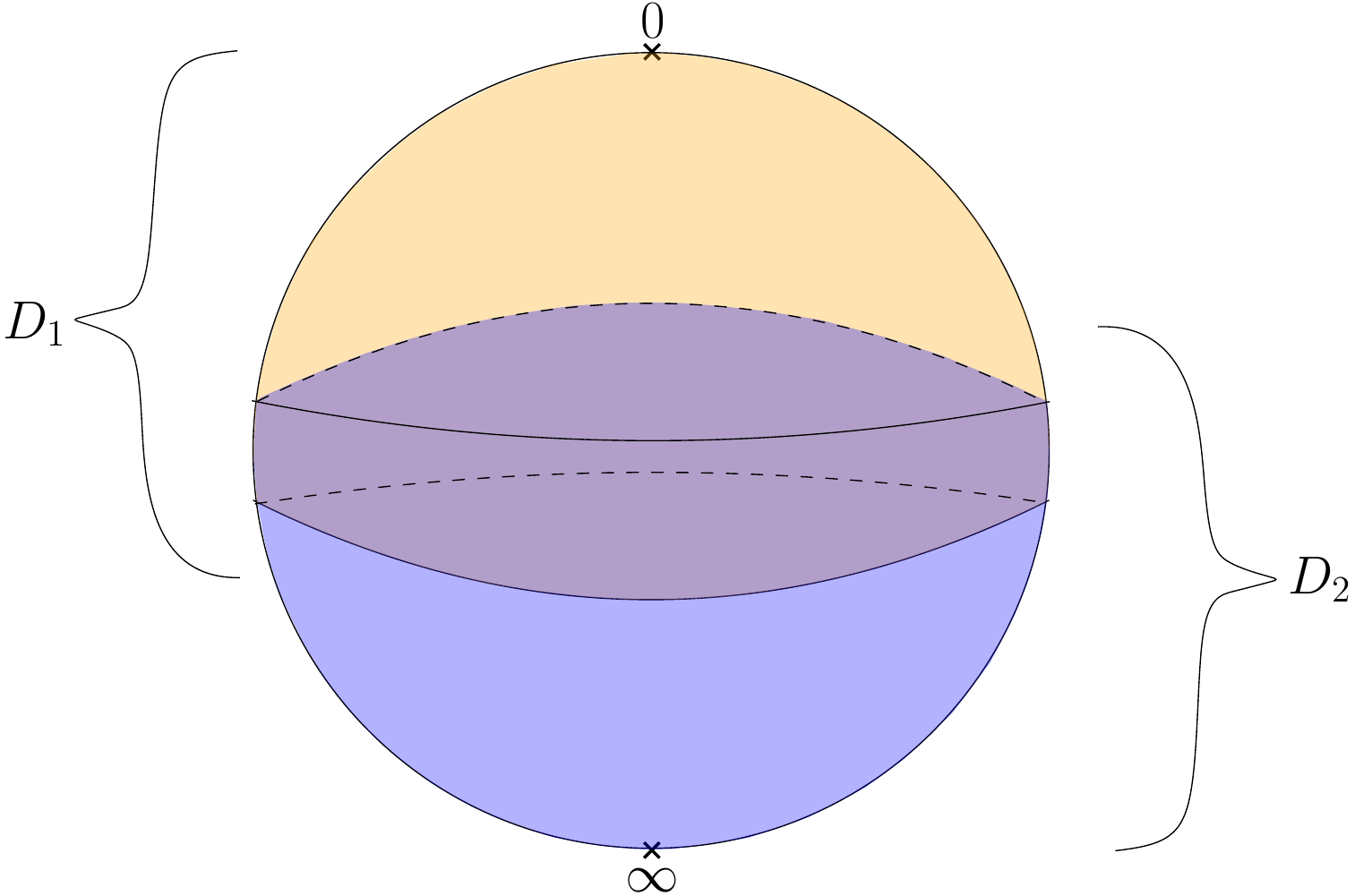}
        \caption{The rigid analytic projective line}
        \end{center}
      \end{figure}

      The marked points of the $G$-covers $f_1, f_2$ are $L$-points in an unramified fiber.
      Their existence ensures that the corresponding field extensions $F_1, F_2$ of $L(z)$ have an unramified prime of degree $1$.
      By \autocite[Lemma 4.2]{haran}, for $i = 1,2$, we can then replace $f_i$ by an isomorphic $L$-$G_i$-cover such that $F_i$ is included in $Q'_i$, and in particular the branch locus $\tbar_i \in \Conf_{r_i}(L)$ of $f_i$ is included in a disk strictly smaller than $D_i$.
      \autocite[Proposition 4.3]{haran} implies that $f_1$ and $f_2$ can be patched into a geometrically irreducible $L$-$G$-cover $f$ with an $L$-point.

    \item
      \textbf{Step 3: Restriction of the patched cover $f$ to disks}

      Denote by $F$ the field extension corresponding to $f$, i.e. the \emph{compound} of $F_1$ and $F_2$ in the terminology of \autocite{haran}, which is included in $\hat Q$.
      By \autocite[Lemma 3.6 (b)]{haran}, we have the equalities $F Q_i = F_i Q_i$ (for $i=1,2$) inside $\hat Q$.
      Moreover, the morphism $\Gal(F Q_i \mid Q_i) \rightarrow \Gal(F \mid L(z))$ corresponds to the inclusion $G_i \hookrightarrow G$.
      We sum this up by the following diagram:
      \[\begin{tikzcd}[ampersand replacement=\&]
        \& {\hat Q} \\
        \& {FQ_i = F_iQ_i} \\
        {F} \& {Q_i} \& {F_i} \\
        \& {L(z)}
        \arrow["{G_i}"', no head, from=4-2, to=3-3]
        \arrow["G", no head, from=4-2, to=3-1]
        \arrow[no head, from=4-2, to=3-2]
        \arrow["{G_i}", no head, from=3-2, to=2-2]
        \arrow[no head, from=3-1, to=2-2]
        \arrow[no head, from=3-3, to=2-2]
        \arrow[no head, from=2-2, to=1-2]
      \end{tikzcd}\]

      Geometrically, the equality $F Q_1 = F_1 Q_1$ means that the cover $f_1$ is isomorphic to $f$ as a rigid analytic cover when both are restricted to the unit disk $D_1$, and similarly for $f_2$ and $f$ in restriction to $D_2$.

      In consequence, the branch points of $f$ are given by the configuration $\tbar = \tbar_1 \cup \tbar_2$.
      Let $y$ be the component of $\Hn^*(G,r_1+r_2)_{\bar L}$ containing $f$ (seen as an $\bar L$-point).
      To show that the component $y$ fits, it remains to check that $\Phi_{r_1+r_2}^{-1}(y) \in \ni(x_1,x_2)$.

    \item
      \textbf{Step 4: Admissibility of the special fiber $\bar f$ of the patched cover}

      Since $\tbar_1$ and $\tbar_2$ are included in disks strictly smaller than $D_1, D_2$, each of the configurations $\tbar_1, \tbar_2$ maps to a single element $\bar{a_1}, \bar{a_2}$ modulo the maximal ideal of $\mathcal{O}$, with $\bar{a_1} \neq \bar{a_2}$.
      The projective line $\P^1_L$ marked by $\tbar = \tbar_1 \cup \tbar_2$ has a semistable model $\tilde P_{\tbar}$ over $\mathcal{O}$, whose special fiber $\bar P_{\tbar}$ is a ``comb'' with two teeth $T_1, T_2$, one for each coset $\bar{a_1}, \bar{a_2}$.
      For $i=1,2$, the points of the configuration $\tbar_i$ extend to sections which specialize to $r_i$ distinct nonsingular points of the tooth $T_i$.
      \begin{figure}[H]
        \begin{center}
        \includegraphics[scale=0.6]{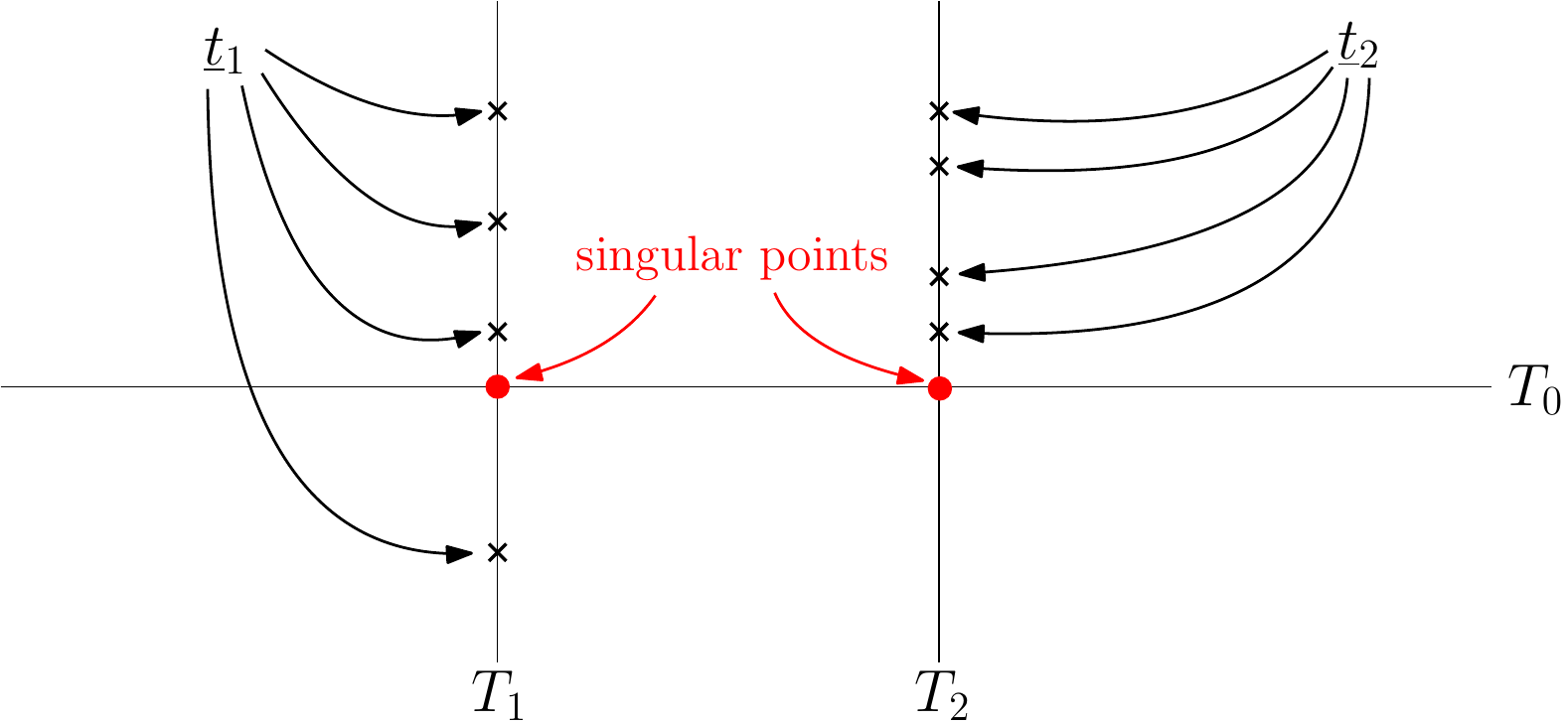}
        \caption{The comb with two teeth $\bar P_{\tbar}$}
        \end{center}
      \end{figure}

      The cover $f$, branched at $\tbar$, extends to a cover $\tilde f$ of the semistable model $\tilde P_{\tbar}$, which is ramified along the sections of the points in $\tbar$.
      The special fiber $\bar f$ of $\tilde f$ is a cover of the comb which lies on the ``boundary'' of the component $y$ in the sense of the Wewers' compactification, see \autocite[Paragraph 1.2]{DebEms} or \autocite[Paragraph 3.3.1]{Cau}.

      To prove that the special fiber $\bar f$ of $\tilde f$ is unramified at the singular points of the comb, we follow \autocite[Paragraph 2.3]{DebEms} closely.
      The restriction of $f$ to $D_1$ extends to a cover (namely, $f_1$) of the rigid projective line which has no branch points outside $D_1$.
      By the arguments of \autocite[Proposition 2.3, (ii)$\Rightarrow$(i)$\Rightarrow$(iii)]{DebEms}, the restricted cover $f_{|D_1}$ is trivial above the annulus $\partial D_1$.
      The same holds for $f_{|D_2}$.
      Hence, $\bar{f}$ is unramified at the singular points of the comb.

      We conclude that $\bar{f}$ is a cover of the comb $\bar P_{\tbar}$ unramified at the singular points, whose restriction to the $i$-th tooth is isomorphic to the cover $f_i$ -- which belongs to the component $x'_i$.

    \item
      \textbf{Step 5: Conclusion}

      The conclusion of Step 4 implies that $\bar{f}$ is a $\Delta$-admissible cover in the sense of \autocite[Definition 3.7]{Cau}, where:
        \[ \Delta = \Big( G, (G_1, G_2), \left (x'_1, x'_2 \right ) \Big) \]
      is the degenerescence structure associated to $\left (x'_1, x'_2 \right )$.
      By \autocite[Proposition 3.9]{Cau}, the component of $f$ is a $\Delta$-component, which in our terminology means that $\Phi^{-1}_{r_1 + r_2}(y) \in \ni(x_1, x_2)$ as we noted in \Cref{subsn:cauthm}.
      This concludes the proof.
  \end{itemize}
\end{proof}

\subsection{Proof of the theorem}
\label{subsn:proof-prodconj}

\begin{theorem}
  \label{thm:prod-conj-defK}
  Let $x,y \in \Comp(G)$ be components defined over $K$.
  Then $\nid(x,y)$ contains a component defined over $K$.
\end{theorem}

\begin{proof}
  Let $r_1 = \deg(x)$, $r_2=\deg(y)$.
  Since the components $x,y$ are defined over $K$, we fix
  $\left \lbrace
  \begin{matrix}
    \text{a $K$-model } X \subseteq \Hn^*(G,r_1)_K \text{ of } x \\
    \text{a $K$-model } Y \subseteq \Hn^*(G,r_2)_K \text{ of } y
  \end{matrix}
  \right .$.
  Note that $X$ and $Y$ are geometrically irreducible.
  The proof consists of three steps:
  \begin{enumerate}
  \item
    First, we inductively construct two sequences of marked $G$-covers $(f_n)_{n \geq 1}$ and $(g_n)_{n \geq 1}$, as well as a sequence $(K_n)$ of field extensions of $K$ such that:
    \begin{itemize}
      \item
        $K_n$ is linearly disjoint with the Galois closure of $K_1 \cdots K_{n-1}$ over $K$.
      \item
        $f_n$ and $g_n$ are $K_n$-points of $X$ and $Y$ respectively.
    \end{itemize}
    
    For $f_1$ and $g_1$, choose arbitrary $\Qbar$-points of $X$ and $Y$ respectively, and let $K_1$ be the smallest extension of $K$ over which they are both rational.
    
    Assume we have constructed $K_1,\ldots,K_{n-1}$ and $f_1,g_1,\ldots,f_{n-1},g_{n-1}$.
    Let $L_n$ be the Galois closure of $K_1 \cdots K_{n-1}$ over $K$.

    Apply \Cref{lem:magic-lemma} with $L=K$, $L'=L_n$ and $S = X$.
    This yields a field extension $\tilde K_n$ of $K$ such that $\tilde K_n$ and $L_n$ are linearly disjoint over $K$, and a $\tilde K_n$-point $f_n$ of $X$.
    Let $L'_n$ be the Galois closure of $L_n \tilde K_n$.
    Apply once again \Cref{lem:magic-lemma} with  $L = \tilde K_n$, $L' = L'_n$ and $S = Y_{\tilde K_n}$.
    This yields a field extension $K_n$ of $\tilde K_n$ such that $K_n$ and $L'_n$ are linearly disjoint over $\tilde K_n$, and a $K_n$-point $g_n$ of $Y$.
    Finally, replace the $\tilde K_n$-point $f_n$ by $f_n$ seen as a $K_n$-point.

    The inclusions between the fields introduced above are summed up by the following diagram:

    \[\begin{tikzcd}[ampersand replacement=\&]
      \& {L'_n} \&\& {K_n} \\
      {L_n} \&\& {\tilde K_n} \\
      \& K
      \arrow[no head, from=2-3, to=1-2]
      \arrow[no head, from=2-1, to=1-2]
      \arrow[no head, from=2-3, to=1-4]
      \arrow[no head, from=3-2, to=2-1]
      \arrow[no head, from=3-2, to=2-3]
    \end{tikzcd}\]

    By construction, we have $f_n \in X(K_n)$ and $g_n \in Y(K_n)$.
    Now:
    \[
      K_n \cap L_n = K_n \cap \left ( L'_n \cap L_n \right ) = \left ( K_n \cap  L'_n \right ) \cap L_n = \tilde K_n \cap L_n = K.
    \]
    Since $L_n \mid K$ is Galois, this is enough to conclude that $K_n$ and $L_n$ are linearly disjoint over $K$.
    We have verified that the constructed sequences $(f_n)$, $(g_n)$, $(K_n)$ satisfy the desired properties.
  
  \item
    Next, we show that for each $n$ there is a component $z_n \in \nid(x,y)$ defined over $K_n$.

    Denote by $\tilde f_n$ (resp. $\tilde g_n$) the $K_n((X))$-point of $X$ (resp. $Y$) obtained by seeing $f_n \in X(K_n)$ (resp. $g_n \in Y(K_n)$) as a $K_n((X))$-point.
    Since $F = K_n((X))$ is a complete valued field, \Cref{lem:patching-in-ni} implies that there is a component $z_n \in \nid(x,y)$ which has a $K_n((X))$-rational point.
    In particular, the field of definition of $z_n$ is included in $K_n((X)) \cap \Qbar = K_n$.

    We have established that there is a component $z_n \in \nid(x,y)$ defined over $K_n$ for all $n$.

  \item
    Finally, since $\nid(x, y)$ is finite, there must be distinct integers $n,n'$ such that $z_n = z_{n'}$.
    Fix such $n,n'$.
    Then, the field of definition of $z_n$ is included in $K_n \cap K_{n'} = K$.

    This concludes the proof: there is a component $z_n \in \nid(x,y)$ defined over $K$.
  \end{enumerate}
\end{proof}

\begin{example}
  \label{ex:m23}
  The Mathieu group $M_{23}$ is the only sporadic simple group not known to be a Galois group over $\Q$.
  In \autocite[Exemple 3.12]{Cau2}, a component defined over $\Q$ of connected $M_{23}$-covers with $15$ branch points is constructed.
  \Cref{thm:prod-conj-defK} improves upon this result.
  The group $M_{23}$ is generated by two conjugate elements $a, a^{\gamma}$ of order $3$.
  Using GAP:
  \begin{lstlisting}[language=GAP]
    a := (1, 22, 14) (2, 13, 9) (3, 8, 6) (7, 16, 21) (10, 18, 19) (11, 23, 12);
    b := (2, 4, 16) (3, 5, 7) (6, 11, 12) (8, 9, 14) (10, 21, 20) (15, 18, 17);
    StructureDescription(Group(a, b)); # Output: "M23"
    IsConjugate(Group(a, b), a, b); # Output: true
  \end{lstlisting}

  By the conclusions of \Cref{ex:bcl-abelian}, the component $x = (a,a^{-1})$ and its conjugate $x^{\gamma}$ are defined over $\Q$.
  By \Cref{thm:prod-conj-defK}, there are elements $\gamma_1, \gamma_2 \in M_{23}$ such that $x^{\gamma_1}x^{\gamma_2 \gamma}$ is a component defined over $\Q$ of connected $M_{23}$-covers with four branch points.
  The same is true of the component $x x^{\tilde\gamma}$ where $\tilde \gamma = \gamma_1^{-1} \gamma_2 \gamma$.
  However, we know little about $\tilde \gamma \in M_{23}$.
  There are many pairs of generators of $M_{23}$ with orders in $\{2,3,4,6\}$, and consequently many other examples with four branch points.
\end{example}

\begin{example}
  \label{ex:17T7}
  Similarly, the group $G = \mathrm{PSL}_2(16) \rtimes \Z/2\Z$ (labeled ``17T7'' on the Klüners-Malle database and on LMFDB), which is the transitive group of least degree not known to be a Galois group over $\Q$, is generated by two conjugate elements $a,b$ of order $6$:
  \begin{lstlisting}[language=GAP]
    a := (1, 11, 5, 13, 14, 17) (3, 15, 7, 12, 8, 6) (9, 10, 16);
    b := (1, 2, 15, 12, 8, 5) (3, 14, 11,4, 9, 6) (7, 10, 17);
    StructureDescription(Group(a, b)); # Output: "PSL(2,16) : C2"
  \end{lstlisting}
  Like in \Cref{ex:m23}, we conclude by \Cref{thm:prod-conj-defK} that there is a component defined over $\Q$ of connected $G$-covers with $4$ branch points.

  The same holds for any group generated by two elements with orders in $\{2,3,4,6\}$.
  More generally, if one has a finite generating set of $G$ and $m(i)$ is the number of generators of order $i$, then there is a component defined over $\Q$ of connected $G$-covers whose number of branch points is $2m(2) + \sum_{i \geq 3} \varphi(i)m(i)$.
\end{example}

\printbibliography
\end{document}